\documentclass[a4paper, 10pt]{article}
\usepackage{pdfpages}
\usepackage{graphicx}%
\usepackage{multirow}%
\usepackage{amsmath,amssymb,amsfonts}%
\usepackage{amsthm}%
\usepackage{mathrsfs}%
\usepackage[title]{appendix}%
\usepackage{xcolor}%
\usepackage{textcomp}%
\usepackage{manyfoot}%
\usepackage{booktabs}%
\usepackage{algorithm}%
\usepackage{algorithmicx}%
\usepackage{algpseudocode}%
\usepackage{listings}%
\usepackage{amsmath, amsthm, amssymb}
\usepackage[T1]{fontenc}
\usepackage[utf8]{inputenc}
\usepackage{lmodern}
\usepackage[margin=3cm]{geometry}
\usepackage{hyperref}
\usepackage{bbm}
\usepackage[normalem]{ulem}

\usepackage[notref,notcite, final]{showkeys}

\usepackage[T1]{fontenc}

\usepackage{mwe} % new package from Martin scharrer
\usepackage{caption}
\usepackage{amsfonts,mathrsfs}
\usepackage{amsmath, amsthm, amssymb}

\usepackage{float}
\usepackage{subfigure}
\newfloat{figure}{H}{lof}
\newfloat{table}{H}{lot}
\floatname{figure}{\figurename}
\floatname{table}{\tablename}

\usepackage{mathtools}

\hypersetup{linkcolor=blue}

\usepackage{enumerate}

\usepackage{float}
\newfloat{figure}{H}{lof}
\newfloat{table}{H}{lot}
\floatname{figure}{\figurename}
\floatname{table}{\tablename}

\usepackage{xr}
\numberwithin{equation}{section}
\newtheorem{theorem}{Theorem}[section]

\newtheorem{assumption}[theorem]{Assumption}

\newtheorem{example}[theorem]{Example}

\newtheorem{lemma}[theorem]{Lemma}

\newtheorem{proposition}[theorem]{Proposition}
\newtheorem{remark}[theorem]{Remark}

\numberwithin{equation}{section}

\newcommand{\dd}{\mathrm{d}}

\newcommand{\dx}{\mathrm{d} x}

\renewcommand{\d}{\mathrm{d}}
\newcommand{\BC}{\mathrm{BC}}

\newcommand{\R}{\mathbb{R}}

\definecolor{QG}{named}{blue}

\newcommand{\writefoot}[1]{
    \renewcommand{\thefootnote}{}
    \footnotetext{\hspace{-16.5pt}\scriptsize#1}
    \renewcommand{\thefootnote}{\arabic{footnote}}
}
\allowdisplaybreaks

\begin{document}

\writefoot{\small \textbf{AMS subject classifications (2020).} Primary: 92D25; Secondary: 34K30, 34A12, 34A35. \smallskip
}
\writefoot{\small \textbf{Keywords.} 
Transfer processes, population dynamics, cooperative and competitive transfers, integro-differential equations.\smallskip
}
\writefoot{\small \textbf{Acknowledgements:} QG acknowledges support from ANR via the project Indyana under grant agreement ANR-21-CE40-0008.}

\begin{center}
    \begin{minipage}{0.8\textwidth}
	\centering
    \LARGE{\bf Robin Hood model versus Sheriff of Nottingham model: transfers in population dynamics}\bigskip
    \end{minipage}

    \Large
    Quentin Griette\medskip \\
    \normalsize
    {\it 
Université Le Havre Normandie, Normandie Univ.,\\
LMAH UR 3821, 76600 Le Havre, France. 
 \\ }
    {\tt quentin.griette@univ-lehavre.fr}
    \bigskip

    \Large
	Pierre Magal$^\dag$ \writefoot{$^\dag$  Pierre Magal was at the origin of this research and contributed to all results in this study, as well as the writing of the initial manuscript.
    Sadly, he passed away before the reviewing process was complete.}\\
    \normalsize
    {\it Department of Mathematics, Faculty of Arts and Sciences , Beijing
Normal University, Zhuhai, 519087, China.\\
Univ. Bordeaux, IMB, UMR 5251, F- 33400 Talence, France. \\
    CNRS, IMB, UMR 5251, F- 33400 Talence, France.\\ }
    {\tt pierre.magal@u-bordeaux.fr}
	\bigskip

	\today
\end{center}
%\title[Robin Hood model versus Sheriff of Nottingham model: transfers in population dynamics]{}

%\author*[2,3,4]{\fnm{Pierre} \sur{Magal}}\email{pierre.magal@u-bordeaux.fr}

%\affil*[1]{\orgname{Normandie Univ, UNIHAVRE, LMAH, FR-CNRS-3335,   ISCN},  \postcode{76600}, \city{ Le Havre},  \country{France}}

\begin{abstract} We study the problem of transfers in a population structured by a continuous variable corresponding to the quantity being transferred.
    The model takes the form of an integro-differential equations with kernels corresponding to the specific rules of the transfer process.
    We focus our interest on the well-posedness of the Cauchy problem in the space of measures. We characterize transfer kernels that give a continuous semiflow in the space of measures and derive a necessary and sufficient condition for the stability of the space $L^1$ of integrable functions.  We construct some examples of kernels that may be particularly interesting in economic applications. Our model considers blind transfers of economic value (e.g. money) between individuals. The two models are the ``Robin Hood model'', where the richest individual unconditionally gives a fraction of their wealth to the poorest when a transfer occurs, and the other extreme, the ``Sheriff of Nottingham model'', where the richest unconditionally takes a fraction of the poorest's wealth. Between these two extreme cases is a continuum of intermediate models obtained by interpolating the kernels. We illustrate those models with numerical simulations and show that any small fraction of the ``Sheriff of Nottingham'' in the transfer rules leads to a segregated population with extremely poor and extremely rich individuals after some time. Although our study is motivated by economic applications, we believe that this study is a first step towards a better understanding of many transfer phenomena occurring in the life sciences.
\end{abstract}

%\keywords{Transfer processes, population dynamics, cooperative and competitive transfers  }

%\maketitle

\section{Introduction}

Transfer phenomena are fundamental processes in population dynamics, influencing a wide range of biological and social systems. These phenomena manifest in various forms, such as predation, parasitism, cooperation, and sexual reproduction, where the transfer typically involves energy, proteins, or genetic material. Social interactions, including opinion formation and economic exchanges, can also be viewed as transfers, with economic transactions being the primary focus of this study. % We refer to the book by Bellouquid and Delitala \cite{bellouquid2006mathematical} and the review papers by Bellomo, Li, and Maini \cite{bellomo2008foundations}, and Bellomo and Delitala \cite{bellomo2008mathematical}.
%Bellomo and Delitala \cite{bellomo2008mathematical}.

\medskip

%Transfer phenomena arise in many natural processes. For example, one may think of mass or energy exchanges between particles,  transfers of proteins or DNA between cells and bacteria, or simply transfers of money or assets in the economy.  We refer to the book by Bellouquid and Delitala \cite{bellouquid2006mathematical} and the review papers by Bellomo, Li, and Maini \cite{bellomo2008foundations}, and
%Bellomo and Delitala \cite{bellomo2008mathematical}.

Their mathematical formulations have a long history. In the context of mathematical physics,  a fundamental  approach has been proposed by Boltzmann, in which interacting particles are viewed as members of a continuum of population density.  In these models transfer of physical quantities from one particle to another is modulated by a kernel function that specifies the transfer process. Many examples have been explored, and a review is found in Perthame \cite{perthame2002kinetic}.  While these models share a structural similarity with the ones we propose here, Boltzmann-type models include a kinetic term linking the position and speed of particles, which makes their analysis quite intricate; here the models we propose do not include such term and we will analyze them by other methods.
\smallskip

In populations dynamics, the transfer of genetic material has been increasingly studied in the past decades. In Magal and Webb \cite{magal2000mutation} and Magal \cite{Magal2002b}, the authors introduced a model devoted to transfers of genetic material in which parent cells exchange genetic material to form their offspring. While they relied on the example of \textit{Helicobacter pyroli}, these models are applicable to a much wider range of species subject to mutation, selection and recombination. More recently, several authors have studied population dynamics model including a mutation and selection (among which  \cite{Desvillettes-Jabin-Mischler-Raoul-2008,Barles-Mirrahimi-Perthame-2009,Canizo-Carrillo-Cuadrado-2013,Burie-Ducrot-Griette-Richard-2020, Calvez-etal-ESAIM}), and models derived from the \textit{Fisher infinitesimal model} that provides a microscopic basis for transfer models of sexual reproduction \cite{barton2017infinitesimal, Calvez-Lepoutre-Poyato-2024} have also attracted some attention.  Raoul \cite{Raoul-2017, Raoul-2021}, in particular, studied several population models which have a kinetic description of the population dynamics.
\medskip

Some social phenomena can also be studied by the use of transfer models.  
Such a model is sometimes called ``kinetic model'' even in the absence of an actual kinetic term;  the literature around kinetic models for opinion formation is developed \cite{Motsch-Tadmor-2014, Giambiagi_Ferrari-etal-2021, Zanella-2023}. These models have also been proposed to model economic phenomena, where the quantity being transferred is an abstract value (money, stocks, etc.)  \cite{Chakraborti-Chakrabarti-2000, Cordier-Pareschi-Toscani-2005,During-Matthes-Toscani-2009, Bisi-2016}. Toscani \cite{Toscani-2006} used a transfer model closely related to the one introduced here in a model of opinion formation. Cordier, Pareschi and Toscani \cite{Cordier-Pareschi-Toscani-2005} investigate a model in which individuals exchange a fraction of their wealth with a random perturbation (and we recover the Robin Hood model when the random perturbation is set to zero). Pareschi and Toscani \cite{Pareschi-Toscani-2006} studied the tails of the asymptotic distribution for such models. Matthes and Toscani \cite{Matthes-Toscani-2008} extended the model to the case of averaged wealth preservation and proved its convergence to an equilibrium. They also provide a smoothness result and a description of the tails of the stationary distributions. In all these models, debts are not allowed, contrary to our situation.  We also refer to the review of Bisi \cite{Bisi-2016}. Recently, Cao and Motsch \cite{Cao-Motsch-2023a,  Cao-Motsch-2023b} introduced and studied a discrete stochastic model in which individuals exchange quantified units of wealth (e.g. one dollar). 
\medskip

This article studies models inspired by the life sciences with potential economic applications.  Here we use a similar idea to model the transfer of richness between individuals. Before we introduce the mathematical concepts rigorously, let us explain our model's idea in a few words. We study a population of individuals who possess a certain transferable quantity (for example, money) and exchange it according to a rule expressed for two individuals chosen randomly in the population. This rule is encoded in an abstract integration kernel but we also propose explicit examples that we call ``Robin Hood'' (RH) and ``Sheriff of Nottingham'' (SN) model, in reference to the well-known folk tale. In the Robin Hood model, the richer gives a fraction of its wealth to the poorest. On the contrary, in the Sheriff of Nottingham model, the richest takes a fraction of the poorest's wealth, possibly through  debt mechanism (we consider that people's wealth may become negative). We also consider distributed versions of those models, in which the result of the interaction is given by a probability distribution instead of just the exchange of a fixed fraction, and prove the convergence to equilibrium. All those models are conservative, meaning that the total wealth in the population remains constant in time.
\medskip

We introduce several novelties compared to the existing literature. Our first result is the existence of solutions as a continuous homogeneous semiflow on the space of measures, as well as the well-posedness of the problem in a strong sense,  for a rather general class of interaction kernels (Theorem \ref{TH3.3}); when in the existing literature we could consult, the existence of solutions is always understood in a weak sense.  We also provide a simple characterization of kernels that leave positively invariant the $L^1$ space of integrable functions, that is to say the measures that are absolutely continuous for the Lebesgue measure. This invariance may occur even when the kernel possesses some singularities (that is the case for the RH and SN models), as long as the singularities occur in a negligible way for the integration variables. We also prove the invariance of  subspace of measures with finite $p$-moment given the appropriate assumption on the interaction kernel. We apply a well-known kinetic method based on the Fourier distance \cite{Matthes-Toscani-2008, Carrillo-Toscani-2007} to prove the existence and global stability of asymptotic distributions in the case of the distributed Robin Hood model. Finally, we compute the transfer operators corresponding to the RH and SN models and show their well-posedness in $L^1$ as well as the set of measures, by applying our previous results.  
\medskip

The plan of the paper is the following. In Section \ref{Section2} we introduce some notations, the functional setting of the paper, and our main assumptions. In Section \ref{Section3}, we present the RH model, which corresponds to cooperative transfers, and the SN model, which corresponds to competitive transfers. These two examples illustrate the problem and will show how the rules at the individual level can be expressed using a proper kernel $K$.  In Section \ref{Section4}, we prove that under Assumption \ref{ASS1.1}, the operator $B$ maps $\mathcal{M}(I) \times \mathcal{M}(I)$ into $\mathcal{M}(I)$ and is a bounded bi-linear operator. Due to the boundedness of $B$, we will deduce that \eqref{1.1}-\eqref{1.2} generates a unique continuous semiflow on the space of positive measures $\mathcal{M}_+(I)$. In Section \ref{Section5}, we consider the restriction of the system \eqref{1.1}-\eqref{1.2}  to $L^1_+(I)$. 
In Section \ref{Section6}, we run individual-based stochastic simulations of the mixed RH and SN model. The paper is complemented with a Supplementary Materials file (denoted by the letter \ref{SectionA}) in which we recall some classical results about measure theory.

\section{Notations, functional setting, and assumptions}
\label{Section2}
Let $I \subset \R$ be a closed interval of $\R$. Depending on the situation, we may consider in the paper that $I$ is bounded or unbounded.
Let us recall that $\mathcal{B}(I)$  is the $\sigma$-algebra generated by all the open subsets of $I$ is called the  \textbf{Borel $\sigma$-algebra}. A subset $A$ of $I$ that belongs to $\mathcal{B}(I)$ is called a \textbf{Borel set}. Throughout this paper we equip $I$ with the Borel $\sigma$-algebra.

Let $\mathcal{M}(I)$ be the space of measures  on $I$.   We will denote $\mathcal{M}_+(I)$ the set of nonnegative measures on $I$ (the positive cone for $\mathcal{M}(I)$).
It is well known that $\mathcal{M}(I)$ endowed with the norm 
$$
\Vert u \Vert_{\mathcal{M}(I)}=\int_{I}\vert u \vert (dx), \forall u \in \mathcal{M}(I),
$$
is a Banach space. Here $|u|=u^++u^-$ is the total variation of the measure $u$, and $u^+$ and $u^-$ are the positive and negative parts of $u$. 
For $u\in \mathcal{M}(I)$ and $n\geq 1$ (not necessarily an integer), we say that $u$ \textit{has a finite $n$-th moment} if $\int_I |x|^n|u|(\dd x)<+\infty$ and denote by $M_n(u)$ the \textit{$n$-th moment} of $u$  
	\begin{equation}\label{eq:def-moment}
	    M_n(u) := \int_{I} x^n u(\dd x). 
	\end{equation}
	We will denote $\mathcal{M}_n(I)$ (resp. $\mathcal{P}_n(I)$) the set of signed measures (resp. probability measures) with finite $n$-th moment that are supported on $I$ ($I=\mathbb{R}$ is allowed). Equipped with the norm $\Vert u\Vert_{\mathcal{M}_n}:=\int_I (1+|x|^n)|u|(\dd x)$, $\mathcal{M}_n(I)$ is a Banach space that is continuously embedded in $\mathcal{M}(I)$, and $\mathcal{P}_n(I)$ is closed in $\mathcal{M}_n(I)$ for this topology.
	If $u\in \mathcal{P}_2(I)$, the \textit{variance} of $u$ is
	\begin{equation}\label{eq:def-variance}
	    V(u) := \int_{I}\big(x-M_1(u)\big)^2 u(\dd x). 
	\end{equation}
We refer the reader to the Supplementary Materials file for a precise definition and to the book of Bogachev \cite{bogachev2007measure} for elementary notions on measure theory.

%\medskip 
%In the above norm definition, the non-trivial part is to define a unique set of finite positive measures $u^+$ and $u^-$ such that 
%$$
%u(dx)=u^+(dx)-u^-(dx).
%$$
%In order to define the absolute value of a measure, one may first realize that the space of measures $\mathcal{M}(I)$  is defined  as 
%$$
%\mathcal{M}(I)=\mathcal{M}_+(I)-\mathcal{M}_+(I),
%$$ 
%where $\mathcal{M}_+(I)$  is the set of positive measures on $I$. 
%
%\medskip 
%Then by the Hahn-decomposition theorem (see Theorem \ref{thm:A.2}), the positive part $u^+(dx)$  and the negative part of $u^-(dx)$ are uniquely  defined. Then the absolute value of $u$ is defined by 
%$$
%\vert u(dx) \vert =u^+(dx)+u^-(dx).
%$$
%The nontrivial part of defining measures' norms is its absolute value. The reader interested can find more results and references in the Appendix \ref{SectionA}. 

\medskip 
An alternative to define the norm of a measure is the following  (see Proposition \ref{THA8} in the Supplementary Materials)
$$
\Vert u \Vert_{\mathcal{M}(I)}=\sup_{ \phi \in \BC\left(I\right): \Vert \phi \Vert _\infty \leq 1} \int_{I} \phi(x) u(dx) , \forall u \in \mathcal{M}(I). 
$$
A finite measure on an interval $I \subset \R$ (bounded or not) is, therefore, a bounded linear form on $\BC\left(I\right)$, the space of bounded and continuous functions from $I$ to $\R$.  Since  $\mathcal{M}(I)$ endowed with its norm is a Banach space, we deduce that $\mathcal{M}(I)$ is a closed subset of $\BC(I)^\star$. When the interval $I$ is not bounded, Example \ref{EXA9} in the Supplementary Materials shows that 
$$
\mathcal{M}(I) \neq BC(I)^\star.
$$

For any Borel subset $A \subset I$,  the quantity 
$$
\int_A u(t,dx) =u(t,A), 
$$
is the number of individuals having their transferable quantity $x$ in the domain $A \subset I$. Therefore, 
$$
\int_I u(t,dx) =u(t,I),
$$
is the total number of individuals at time $t$. 

\medskip 
The operator of transfer $T:  \mathcal{M}_+(I) \to \mathcal{M}_+(I)$ is defined by 
\begin{equation}\label{eq:defT}
	T\left(u\right)(dx):= \left\{ 
	\begin{array}{ll}
	    \dfrac{B(u,u)(dx)}{\int_{I}u}, &\text{ if } u \in \mathcal{M}_+(I)\setminus \left\{0\right\},  \vspace{0.2cm} \\
		0,  &\text{ if } u=0, 
	\end{array}
	\right. 
\end{equation}
where $B: \mathcal{M}(I) \times \mathcal{M}(I) \to \mathcal{M}(I)$ is a bounded bi-linear map defined by 
\begin{equation*}
	B(u, v)(dx):=\iint_{I^2} K(dx, x_1, x_2)u(\dd x_1) v(\dd x_2),
\end{equation*}
or equivalently, defined by 
\begin{equation*}
	B(u, v)(A):=\iint_{I^2} K(A, x_1, x_2)u(\dd x_1) v(\dd x_2),
\end{equation*}
for each Borel set $A\subset I$, and each $u, v\in \mathcal{M}(I)$; and $\int_{I}u$ is the total mass of $u$.  The assumptions on the kernel $K$ will be made precise later in Assumption \ref{ASS1.1}. Let us stress at this point that, even if we use an integral sign to denote the integral with respect to the measure $u$ and $v$, we do not assume that $u$ or $v$ are absolutely continuous with respect to the Lebesgue measure.

In the economic interpretation of the model, the kernel $K(\dd x, x_1, x_2)$ describes the repartition of wealth after two individuals of wealth $x_1$ and $x_2$ meet. For instance, in the case of the Robin Hood model presented below, the meeting of two individuals of wealth $(x_1, x_2)$ results in two individuals of wealth $\big((1-f) x_1+f x_2, (1-f)x_1 + f x_2)\big)$ for some $f\in(0,1)$ (the Sheriff of Nottingham model is similar but with $f>1$). Note that  meetings are symmetric (i.e. the meeting $(x_2, x_1)$ always occurs at the same time as $(x_1, x_2)$ and the result is the same), hence the $1/2$ factor in \eqref{eq:RH}. Distributed versions will also be considered (see Example \ref{ex:Distributed-RH}); in that case, the result of the interaction between the two individuals is no longer  deterministic and $K(\dd x, x_1, x_2)$ can be interpreted as a probability distribution. We present simulations of the microscopic model corresponding to this economic interpretation in section \ref{Section6}.

\medskip 
Let $T:  \mathcal{M}_+(I) \to \mathcal{M}_+(I)$ be a transfer operator. Let $\tau>0$ be the rate of transfers. We assume the time between two transfers follows an exponential law with a mean value of $1/\tau$. Two individuals will be involved once the transfer time has elapsed, so the transfer rate will have to be doubled (i.e, equal to  $2 \tau$).  The model of transfers is an ordinary differential equation on the space of positive measures 
\begin{equation} \label{1.1}
	\partial_t u(t, dx) = 2 \tau \, T\big(  u(t,.)\big)(dx) - 2 \tau \, u(t, dx),  \forall t \geq 0.
\end{equation}
\medskip 
The equation \eqref{1.1} should be complemented with the measure-valued  initial distribution 
\begin{equation} \label{1.2}
	u(0,dx)=\phi(dx) \in \mathcal{M}_+(I).  
\end{equation}
Let us stress at this point that $u(0, \dd x)$ need not be absolutely continuous with respect to the Lebesgue measure and may well possess singularities such as Dirac masses (and so does $u(t, \dd x)$). 

A solution of \eqref{1.1}, will be a continuous function $u:[0, +\infty) \to \mathcal{M}_+(I)$, satisfying the fixed point problem 
\begin{equation} \label{1.3}
	u(t,dx)=\phi(dx)+ \int_0^t  2\tau \, T\big(  u(\sigma,.)\big)(dx) - 2\tau \, u(\sigma, dx) d\sigma, 
\end{equation}
where the above integral is a Riemann integral in $(\mathcal{M}(I), \Vert . \Vert_{\mathcal{M}(I)})$.  

\medskip 
To prove the positivity of the solution \eqref{1.1}, one may prefer to use the fixed point problem  
\begin{equation} \label{1.4}
	u(t,dx)=e^{-2\tau t}\phi(dx)+ \int_0^t  e^{-2\tau \left(t-s\right)}2\tau \, T\big(  u(\sigma,.)\big)(dx) d\sigma, 
\end{equation}
which is equivalent to \eqref{1.3}. 

\medskip 
For each $t \geq 0$, 
$$
u(t,dx) \in \mathcal{M}_+(I), 
$$ 
is understood as a measure-valued population density at time $t$. 

\medskip 

\medskip 
Based on the examples of kernels presented in Section \ref{Section3}, we can make the following assumptions. 
\begin{assumption}\label{ASS1.1}
	We assume the kernel $K$ satisfies the following properties  
	\begin{enumerate}
		\item[{\rm (i)}] For each Borel set $A\in\mathcal{B}(I)$, the map $(x_1, x_2)\mapsto K(A, x_1, x_2)$ is Borel measurable.
		\item[{\rm (ii)}]  For each $(x_1, x_2)\in I \times I$, the map $A\in\mathcal{B}(I)\mapsto K(A, x_1, x_2)$ is a probability measure.
		%\item[{\rm (iii)}]  For each $(x_1, x_2)\in I \times I$, 
		%$$
		%\int_{I}x \, K(dx, x_1,x_2)=\dfrac{x_1+x_2}{2}.
		%$$
	\end{enumerate}
\end{assumption}
Assumption \ref{ASS1.1} is relatively weak and can be satisfied by many kernels. In particular, it is the case for kernels $K(dz, x_1, x_2)=K(z, x_1, x_2)\dd z$ that are given by a nonnegative integrable function of $\mathbb{R}^3$ (with respect to the Lebesgue measure).

We will often associate Assumption \ref{ASS1.1} with  the following.
\begin{assumption}\label{ASS1.2}
    We assume that $\int |x|K(\dd x, x_1, x_2)<+\infty$ for all $x_1, x_2\in I$, and that  there exists a constant $C$ such that 
    \begin{equation*}
	\int |x| K(\dd x, x_1, x_2)\leq C\big(|x_1|+|x_2|\big).
    \end{equation*}
Moreover, for each $(x_1, x_2)\in I \times I$, we assume that 
		$$
		\int_{I}x \, K(\dd x, x_1,x_2)=\dfrac{x_1+x_2}{2}.
		$$
\end{assumption}

\section{Examples of transfer kernels}	
\label{Section3}
Constructing a transfer kernel is a difficult problem. In this section, we propose two examples of transfer models. Some correspond to existing examples in the literature, while others seem new.

\subsection{Robin Hood model: (the richest give to the poorest)}

The poorest gains a fraction $f$ of the difference $\vert x_2-x_1 \vert $, and the richest looses a fraction $f$ of the difference $\vert x_2-x_1 \vert $
\begin{equation}\label{eq:RH}
	K_1(\dd x, x_1, x_2) := \frac{1}{2}\left(\delta_{x_2 - f( x_2-x_1)}(\dd x) + \delta_{x_1 - f(x_1-x_2)}(\dd x)\right), 
\end{equation}
where $f\in(0,1)$ is fixed. 

\medskip 
In Figure \ref{Fig1}, we explain why we need to consider a mean value of two Dirac masses in the kernel $K_1$. Consider a transfer between two individuals, and consider $x_1$ and $x_2$ (respectively $y_1$ and $y_2$) the values of the transferable quantities before transfer (respectively after transfer). When we define  $K_1$ we need to take a mean value, because  
we can not distinguish if $x_1$ and $x_2$ are ancestors of  $y_1$ or $y_2$. In other words, we can not distinguish between the two cases:  1) $x_1 \to y_1$ and $x_2 \to y_2$; and 2)  $x_1 \to y_2$ and $x_2 \to y_1$; (where $x \to y$ means $x$ becomes $y$).

\begin{figure}
	\centering
	\includegraphics[scale=0.4]{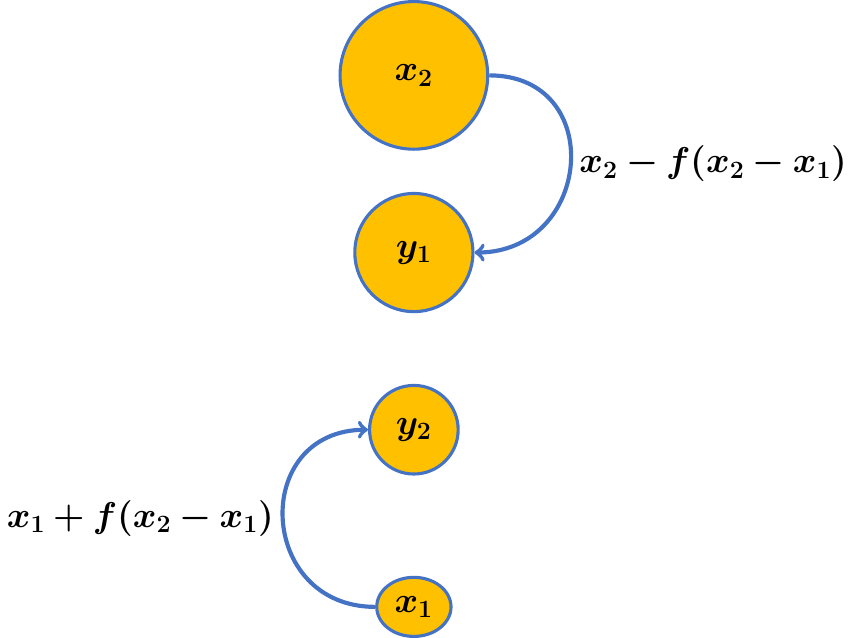}\includegraphics[scale=0.4]{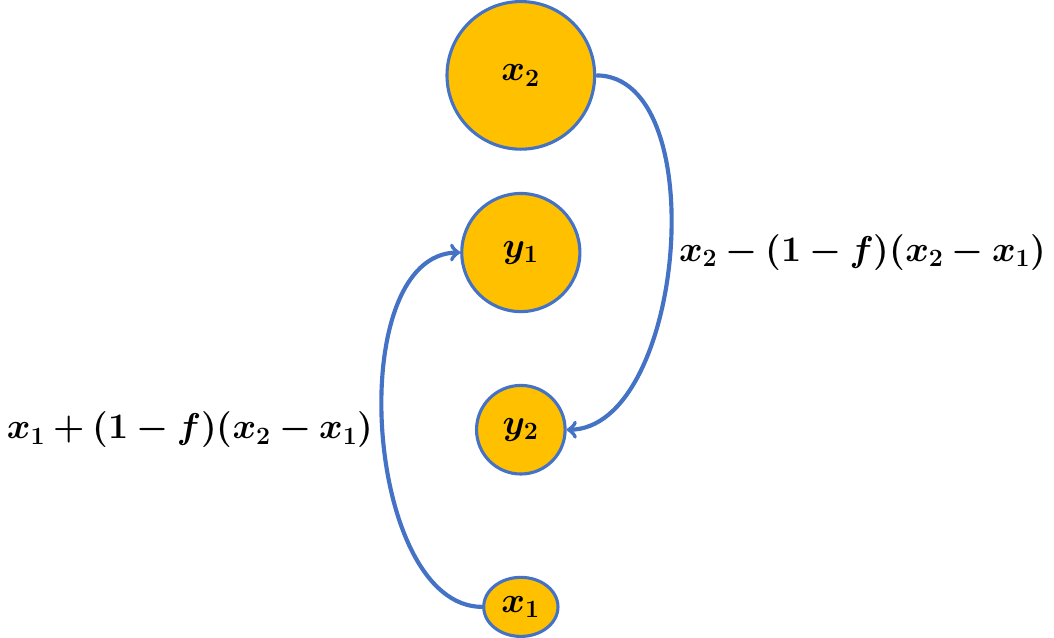}
	\caption{{\textit{In the figure, a transfer between two individuals, and consider $x_1$ and $x_2$ (respectively $y_1$ and $y_2$) the values of the transferable quantities before transfer (respectively after transfer).  We plot a transfer whenever $x_2>x_1$, and $f\in [0,1/2]$ (on the left hand side), and $(1-f)\in [1/2,1]$ (on the right hand side). We observe that the values $y_1$ and $y_2$  are the same on both sides. }}}\label{Fig1}
\end{figure}
In Figure \ref{Fig1},  we use the following equalities
\begin{equation*}
	y_1=	x_2-f (x_2-x_1)=x_1+(1-f) (x_2-x_1),
\end{equation*}
and 
\begin{equation*}
	y_2=x_2-(1-f) (x_2-x_1)=x_1+(1-f) (x_2-x_1).
\end{equation*}
The values $y_1$ and $y_2$  are the same after a given transfer if we choose  $f$ or $1-f$. In other words, we can not distinguish if $x_1$ and $x_2$ are ancestors of  $y_1$ or $y_2$. This explains the mean value of Dirac masses in the kernel $K_1$.   It follows, that  this kind of kernel will preserve the convex hull of the support of the initial distribution, so we can restrict to any closed bounded interval $I \subset \R$. This transfer kernel corresponds to the one proposed to build the recombination operator in Magal and Webb \cite{magal2000mutation}. An extended version with friction was proposed by Hinow, Le Foll, Magal, and Webb \cite{hinow2009analysis}. 

Here we can compute $B_1(u, v)$ explicitly when $I=\mathbb{R}$,  $u\in L^1(\mathbb{R})$ and $v\in L^1(\mathbb{R})$. Indeed let $\varphi\in C_c(I)$ be a compactly supported test function. Then
\begin{align*}
	\int_{\mathbb{R}}\varphi(x) B_1(u, v)(\dd x) &=\int_{x_1\in \mathbb{R}}\int_{x_2\in \mathbb{R}} \int_{x\in \mathbb{R}} \varphi(x) K_1(\dd x, x_1, x_2)u(x_1)\dd x_1 u(x_2)\dd x_2 \\ 
	&= \iint_{\mathbb{R}\times \mathbb{R}}  \frac{1}{2}\left(\varphi\big(x_2-f(x_2-x_1)\big)+\varphi\big(x_1-f(x_1-x_2)\big)\right) u(x_1)v(x_2)\dd x_1 \dd x_2 \\
	&=\frac{1}{2}\left(\iint_{\mathbb{R}\times \mathbb{R}} \varphi\big((1-f)x_2+fx_1\big) u(x_1)v(x_2)\dd x_1\dd x_2\right. \\
	&\quad\left.+\iint_{\mathbb{R}\times \mathbb{R}}\varphi\big((1-f)x_1+fx_2\big) u(x_1)v(x_2)\dd x_1\dd x_2\right)\\
	&=\frac{1}{2}\left(\iint_{\mathbb{R}\times \mathbb{R}} \varphi\big((1-f)x_2+fx_1\big) u(x_1)v(x_2)\dd x_1\dd x_2\right. \\
	&\quad\left.+\iint_{\mathbb{R}\times \mathbb{R}}\varphi\big((1-f)x_2+fx_1\big) u(x_2)v(x_1)\dd x_1\dd x_2\right).\\
\end{align*}
Next, by using the change of variable 
$$
\left\{ 
\begin{array}{l}
	x_1= x-(1-f) \sigma\\
	x_2=x+f\sigma 
\end{array}
\right.
\Leftrightarrow
\left\{ 
\begin{array}{l}
	\sigma=  x_2-x_1\\
	x=(1-f)x_2+f x_1
\end{array}
\right. 
% \left( \sigma , x \right)= \left( x_2-x_1, (1-f)x_2+f x_1\right),
$$ we obtain 
\begin{align*}
	\int_{\mathbb{R}}\varphi(x) B_1(u, v)(\dd x) &=\frac{1}{2}\left(\int_{\mathbb{R}} \varphi\big(x\big)\int_{\mathbb{R}}  u(x-(1-f) \sigma)v(x+f\sigma)\dd \sigma\dd x\right. \\
	&\quad\left.+\int_{\mathbb{R}} \varphi\big(x\big)\int_{\mathbb{R}} u(x+f\sigma) v(x-(1-f) \sigma)\dd \sigma \dd x\right),
\end{align*}
and we obtain 
\begin{equation*}
	B_1(u,v)(x)= \dfrac{1}{2}\left(\int_{\mathbb{R}} u(x-(1-f)\sigma)v\left(x+f\sigma \right)\dd \sigma + \int_{\mathbb{R}} v(x-(1-f)\sigma)u\left(x+f \sigma \right)\dd \sigma\right). 
\end{equation*}
We conclude that the transfer operator restricted to $L^1_+(\R)$ is defined by 
\begin{equation*}
	T_1(u)(x)=
	\left\{ 
	\begin{array}{ll}
		\dfrac{\int_{\mathbb{R}} u(x-(1-f)\sigma)u\left(x+f\sigma \right) \dd \sigma}{\int_{\mathbb{R}} u(x)\dd x} , \text{ if } u \in L^1_+(\R) \setminus \left\{0\right\}, \vspace{0.2cm} \\
		0, \text{ if } u =0. 
	\end{array}		
	\right.
\end{equation*}

\begin{remark}
	One may also consider the case where the fraction transferred varies in function  of the distance between the poorest and richest before transferred. This problem was considered by Hinow, Le Foll, Magal and Webb \cite{hinow2009analysis}, and in the case 
	\begin{equation*}
		K_1(\dd x, x_1, x_2) := \frac{1}{2}\left(\delta_{x_2 - f(\vert x_2-x_1\vert)( x_2-x_1)}(\dd x) + \delta_{x_1 -  f(\vert x_2-x_1\vert)(x_1-x_2)}(\dd x)\right), 
	\end{equation*}
	where $f:[0, +\infty) \to [0,1]$ is a continuous function. 
\end{remark}
\subsection{Sheriff of Nottingham model: (the poorest give to the richest)}
\label{sec:Sheriff-Nottingham}
The poorest looses a fraction $f$ of the difference $\vert x_2-x_1 \vert $, and the richest gains a fraction $f$ of the difference $\vert x_2-x_1 \vert $
\begin{equation}\label{eq:Sheriff-Nottingham}
	K_2(\dd x, x_1, x_2) := \frac{1}{2}\left(\delta_{x_2 +f( x_2-x_1)}(\dd x) + \delta_{x_1 + f(x_1-x_2)}(\dd x)\right), 
\end{equation}
where $f\in(0,1)$ is fixed.  

\medskip 
This kind of kernel will expand the support of the initial distribution to the whole real line, therefore we can not restrict to bounded intervals $I \subset \R$. \textbf{In this case, the only possible choice for the interval is $I=\mathbb{R}$.}

\medskip 
In Figure \ref{Fig2}, we explain why we need to consider a mean value of two Dirac masses in the kernel $K_1$. Indeed, we have 
$$
y_2=	x_2+f (x_2-x_1)=x_1+(1+f) (x_2-x_1),
$$
and 
$$
y_1=	x_2-(1+f) (x_2-x_1)=x_1-f (x_2-x_1).
$$
Therefore, the transferable quantities $y_1$ and $y_2$ after a given transfer, are the same with $f$ or $1+f$. 
\begin{figure}
	\centering
	\includegraphics[scale=0.4]{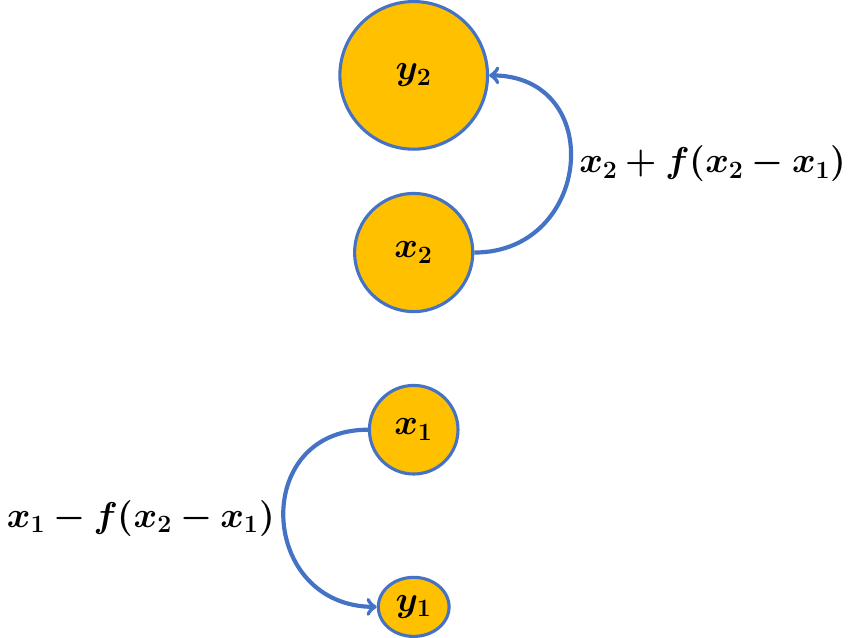}\includegraphics[scale=0.4]{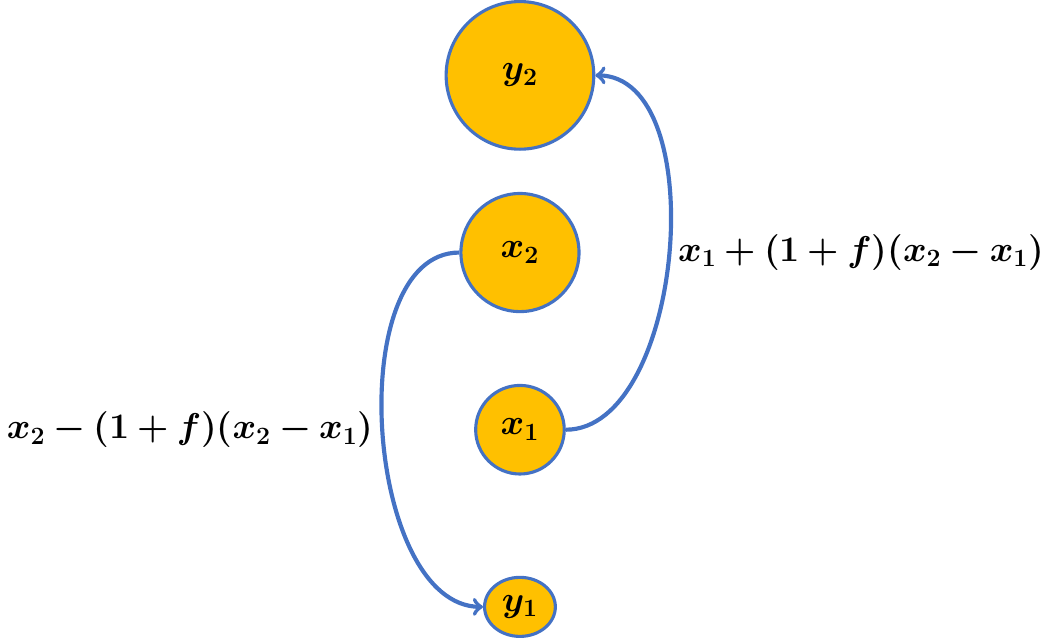}
	\caption{{\textit{In the figure, the two values before transfers are $x_1$ and $x_2$, and the two values after transfer are $y_1$ and $y_2$. We plot a transfer whenever $x_2>x_1$, and $f\in [0,1]$ (on the left hand side), and $f$ is replaced by $1+f$ (on the right hand side). We observe that the values $y_1$ and $y_2$  are the same on both sides. }}}\label{Fig2}
\end{figure}

Here again we can compute $B_2(u, v)$ explicitly when $I=\mathbb{R}$,  $u\in L^1(\mathbb{R})$ and $v\in L^1(\mathbb{R})$. Indeed let $\varphi\in C_c(I)$ be a compactly supported test function. Then
\begin{align*}
	\int_{\mathbb{R}}\varphi(x) B_2(u, v)(\dd x) 
	&=\int_{x_1\in \mathbb{R}}\int_{x_2\in \mathbb{R}} \int_{x\in \mathbb{R}} \varphi(x) K_2(\dd x, x_1, x_2)u(x_1)\dd x_1 u(x_2)\dd x_2 \\ 
	&= \iint_{\mathbb{R}\times \mathbb{R}}  \frac{1}{2}\left(\varphi\big(x_2+f(x_2-x_1)\big)+\varphi\big(x_1+f(x_1-x_2)\big)\right) u(x_1)v(x_2)\dd x_1 \dd x_2 \\
	&=\frac{1}{2}\left(\iint_{\mathbb{R}\times \mathbb{R}} \varphi\big((1+f)x_2-fx_1\big) u(x_1)v(x_2)\dd x_1\dd x_2\right. \\
	&\quad\left.+\iint_{\mathbb{R}\times \mathbb{R}}\varphi\big((1+f)x_1-fx_2\big) u(x_1)v(x_2)\dd x_1\dd x_2\right) \\
	&=\frac{1}{2}\left(\iint_{\mathbb{R}\times \mathbb{R}} \varphi\big((1+f)x_2-fx_1\big) u(x_1)v(x_2)\dd x_1\dd x_2\right. \\
	&\quad\left.+\iint_{\mathbb{R}\times \mathbb{R}}\varphi\big((1+f)x_2-fx_1\big) u(x_2)v(x_1)\dd x_1\dd x_2\right).
\end{align*}
Next, by using the change of variable 
$$
\left\{ 
\begin{array}{l}
	x_1= x-(1+f) \sigma\\
	x_2=x-f\sigma 
\end{array}
\right.
\Leftrightarrow
\left\{ 
\begin{array}{l}
	\sigma=  x_2-x_1\\
	x=(1+f)x_2-f x_1
\end{array}
\right. 
% \left( \sigma , x \right)= \left( x_2-x_1, (1-f)x_2+f x_1\right),
$$ we obtain 
\begin{align*}
	\int_{\mathbb{R}}\varphi(x) B_2(u, v)(\dd x) &=\frac{1}{2}\left(\int_{\mathbb{R}} \varphi\big(x\big)\int_{\mathbb{R}}  u(x-(1+f) \sigma)v(x-f\sigma)\dd \sigma\dd x\right. \\
	&\quad\left.+\int_{\mathbb{R}} \varphi\big(x\big)\int_{\mathbb{R}} u(x-f\sigma) v(x-(1+f) \sigma)\dd \sigma \dd x\right),
\end{align*}
and we obtain 
\begin{equation*}
	B_2(u,v)(x)= \dfrac{1}{2}\left(\int_{\mathbb{R}} u(x-(1+f)\sigma)v\left(x-f\sigma \right)\dd \sigma + \int_{\mathbb{R}} v(x-(1+f)\sigma)u\left(x-f \sigma \right)\dd \sigma\right). 
\end{equation*}
%We conclude that the transfer operator restricted to $L^1_+(\R)$ is defined by 
%	\begin{equation*}
	%	T_2(u)(x)=\frac{1}{\int_{\mathbb{R}} u(\hat{x})\dd\hat{x}} \int_{\mathbb{R}}u(x-f\sigma-\sigma)u\left(x-f \sigma \right)\dd \sigma.
	%\end{equation*}
	We conclude that the transfer operator restricted to $L^1_+(\R)$ is defined by 
	\begin{equation*}
		T_2(u)(x)=
		\left\{ 
		\begin{array}{ll}
			\dfrac{\int_{\mathbb{R}} u(x-(1+f)\sigma)u\left(x-f\sigma \right) \dd \sigma}{\int_{\mathbb{R}} u(x)\dd x} , \text{ if } u \in L^1_+(\R) \setminus \left\{0\right\}, \vspace{0.2cm} \\
			0, \text{ if } u =0. 
		\end{array}		
		\right.
	\end{equation*}
	
	\section{Understanding \eqref{1.1} in the space of measures}   
	\label{Section4}
	Before anything, we need to define $B(u, v)(\dd x)$ whenever $u$ and $v$ are finite measures on $I$. 
	
	\begin{theorem} \label{TH3.1}
		Let Assumption \ref{ASS1.1} be satisfied. 
		Define
		\begin{equation}\label{3.1}
			B(u, v)(A):=\iint_{I^2} K(A, x_1, x_2)u(\dd x_1) v(\dd x_2),
		\end{equation}
		for each Borel set $A\subset I$, and each $u, v\in \mathcal{M}(I)$. 
		
		\medskip 
		\noindent 	Then $B$ maps $\mathcal{M}(I) \times \mathcal{M}(I)$ into $\mathcal{M}(I)$, and satisfies the following properties 
		\begin{itemize}
			\item[{\rm (i)}]
			\begin{equation*}
				\Vert B(u, v)\Vert_{\mathcal{M}(I)} \leq  \Vert u\Vert_{\mathcal{M}(I)} \Vert v\Vert_{\mathcal{M}(I)}, \forall u, v\in \mathcal{M}(I).
			\end{equation*}
			\item[{\rm (ii)}]
			\begin{equation*}
			    B(u, v)\in \mathcal{M}_+(I), \forall u, v\in \mathcal{M}_+(I),
			\end{equation*}
			\item[{\rm (iii)}]
			\begin{equation*} 
			    \int_{I}B(u, v)(\dd x) = \int_{I}u(\dd x_1) \int_{I}v(\dd x_2), \forall u,v\in \mathcal{M}_+(I).
			\end{equation*}
		\end{itemize}
		If moreover Assumption \ref{ASS1.2} is satisfied, then the following property also holds. 
		\begin{itemize}
			\item[{\rm (iv)}] 	
			\begin{equation*} 
			    \int_{I}x B(u, v)(\dd x) = \dfrac{\int_{I}x_1 u(\dd x_1) \int_{I}v(\dd x_2)+\int_{I}u(\dd x) \int_{I}x_2 v(\dd x_2) }{2}, \forall u,v\in \mathcal{M}_+(I).
			\end{equation*}
			%\item[{\rm (v)}] For each integer $n \geq 0$, 
			%\begin{equation*}
			%	\int_{I}x^n B(u,v)(dx)=\int_{I \times I} \int_{I}x^n K(dx,x_1,x_2) \, u(dx_1) \, v(dx_2), \forall u,v\in \mathcal{M}_+(I).
			%\end{equation*}
		\end{itemize}
	\end{theorem}
	\begin{proof}
		Let $u\in \mathcal{M}(I)$, $v\in \mathcal{M}(I)$ and define
		$$
		w(\dx) :=B(u, v)(\dx) 
		$$ 
		by \eqref{3.1}.
		Let $(A_n)_{n\in\mathbb{N}}$ a collection of pairwise disjoint Borel-measurable sets in $I$.
		We want to prove that 
		\begin{align}\label{3.2}
			w\left(\bigcup_{n\in\mathbb{N}} A_n\right)&=  \sum_{n\in\mathbb{N}}w(A_n).
		\end{align}
		We have 
		\begin{equation*}
			\begin{array}{ll}
				w\left(\bigcup_{n\in\mathbb{N}} A_n\right) = \displaystyle \iint_{I^2} K\left(\bigcup_{n\in\mathbb{N}} A_n, x_1, x_2\right)u(\dd x_1) v(\dd x_2)
				 = \displaystyle  \iint_{I^2} \sum_{n\in\mathbb{N}}K(A_n, x_1, x_2)u(\dd x_1) v(\dd x_2). 
			\end{array}
		\end{equation*}
		In order to change the order of summation between the integral and  the sum, we will use Fubini's Theorem \cite[Vol. I Theorem 3.4.4 p.185]{bogachev2007measure} and Tonnelli's Theorem \cite[Vol. I Theorem 3.4.5 p.185]{bogachev2007measure}.

		We consider $P\subset \mathcal{B}(I^2)$ the support of the positive part of $u\otimes v$ (as given by \cite[Theorem 3.1.1 p. 175]{bogachev2007measure}). That is 
		$$
		\mathbbm{1}_P u\otimes v\in \mathcal{M}_+(I^2), \text{ and }-\mathbbm{1}_{P^c} u\otimes v\in \mathcal{M}_+(I^2), 
		$$
		where    $\mathbbm{1}_P$ (respectively  $\mathbbm{1}_{P^c}$) is the indicator functions of $P$, that is $\mathbbm{1}_P (x)=1$ if $x \in  P$ else $\mathbbm{1}_P (x)=0$   (respectively the indicator function of  $P^c=I \setminus P$ the complement set of $P$).  
		
		\medskip 
		We consider the maps defined for all $n\in\mathbb{N}$,  and all $x_1, x_2\in I,$
		\begin{equation}
			f_n(x_1, x_2):=\mathbbm{1}_PK(A_n, x_1, x_2),
		\end{equation}
		and
		\begin{equation*}
			\, f(n, x_1, x_2):=f_n(x_1, x_2).
		\end{equation*}
		Then by Assumption \ref{ASS1.1}-(i), for each integer $n\in\mathbb{N}$ the map $f_n$ is Borel measurable (i.e., measurable with respect to the Borel $\sigma$-algebra), and for each Borel set $B\subset\mathbb{R}$, we have
		\begin{equation*}
			f^{-1}(B) = \bigcup_{n\in\mathbb{N}} \{n\}\times f_n^{-1}(B). 
		\end{equation*}
		Consequently, $f$ is measurable for the $\sigma$-algebra $\mathcal{P}(\mathbb{N})\otimes \mathcal{B}(I^2)$, the smallest $\sigma$-algebra in $\mathcal{P}(\mathbb{N}\times I^2)$ that contains all rectangles $\mathcal{N}\times B$ where $\mathcal{N}\in \mathcal{P}(\mathbb{N})$ and $B\in\mathcal{B}(I^2)$.
		
		\medskip 
		Let $c$ be the counting measure on $\mathbb{N}$,  $c:=\sum_{n\in\mathbb{N}}\delta_n$. By Tonnelli's Theorem \cite[Vol. I Theorem 3.4.5 p.185]{bogachev2007measure}, since $f$ is nonnegative, and $c$ and $(u\otimes v)_+$ are nonnegative $\sigma$-finite measures,  and
		\begin{align*}
			\int_{I^2}\int_{\mathbb{N}} f(n, x_1, x_2)c(\dd n) (u\otimes v)_+(\dd x_1 \dd x_2) 
			&= \int_{I^2} \sum_{n\in\mathbb{N}}\mathbbm{1}_P(x_1, x_2)K(A_n, x_1, x_2)(u \otimes v)_+(\dd x_1 \dd x_2)  \\ 
			&=\iint_{I^2} \mathbbm{1}_P(x_1, x_2)K\left(\bigcup_{n\in\mathbb{N}}A_n, x_1, x_2\right)u(\dd x_1) v(\dd x_2) \\ 
			&\leq  \iint_{I^2} \mathbbm{1}_P(x_1, x_2)K\left(I^2, x_1, x_2\right)u(\dd x_1) v(\dd x_2) \\
			& = (u\otimes v)(P)<+\infty.
		\end{align*}
		We conclude that $f\in L^1(c\otimes (u\otimes v)_+) $ and therefore by Fubini's Theorem \cite[Vol. I Theorem 3.4.4 p.185]{bogachev2007measure} we have 
		\begin{align*}
				\int_{I^2}\int_{\mathbb{N}} f(n, x_1, x_2)c(\dd n) (u\otimes v)_+(\dd x_1 \dd x_2)  
			= \int_{\mathbb{N}}\int_{I^2} f(n, x_1, x_2)c(\dd n) (u\otimes v)_+(\dd x_1 \dd x_2). 
		\end{align*}
		This means that 
		\begin{align*}
			\iint_{I^2} \sum_{n\in\mathbb{N}} \mathbbm{1}_P K(A_n, x_1, x_2)u(\dd x_1) v(\dd x_2) 
			= \sum_{n\in\mathbb{N}}\iint_{I^2} \mathbbm{1}_P K(A_n, x_1, x_2)u(\dd x_1) v(\dd x_2). 
		\end{align*}
		By similar arguments we can show that 
		\begin{align*}
			\iint_{I^2} \sum_{n\in\mathbb{N}} (-\mathbbm{1}_{P^c}) K(A_n, x_1, x_2)u(\dd x_1) v(\dd x_2) 
			= \sum_{n\in\mathbb{N}}\iint_{I^2} (-\mathbbm{1}_{P^c}) K(A_n, x_1, x_2)u(\dd x_1) v(\dd x_2). 
		\end{align*}
		Thus we obtain
		\begin{align*}
		    \iint_{I^2} \sum_{n\in\mathbb{N}} &K(A_n, x_1, x_2)u(\dd x_1) v(\dd x_2)
			=\iint_{I^2} \sum_{n\in\mathbb{N}}(\mathbbm{1}_P+\mathbbm{1}_{P^c}) K(A_n, x_1, x_2)u(\dd x_1) v(\dd x_2)  \\ 
			&=\iint_{I^2} \sum_{n\in\mathbb{N}}\mathbbm{1}_P K(A_n, x_1, x_2)u(\dd x_1) v(\dd x_2)  
			+ \iint_{I^2} \sum_{n\in\mathbb{N}}\mathbbm{1}_{P^c} K(A_n, x_1, x_2)u(\dd x_1) v(\dd x_2)\\
			&=\sum_{n\in\mathbb{N}}\iint_{I^2} \mathbbm{1}_P K(A_n, x_1, x_2)u(\dd x_1) v(\dd x_2)  
			 + \sum_{n\in\mathbb{N}}\iint_{I^2} \mathbbm{1}_{P^c} K(A_n, x_1, x_2)u(\dd x_1) v(\dd x_2)\\
			&=\sum_{n\in\mathbb{N}} \iint_{I^2} K(A_n, x_1, x_2)u(\dd x_1) v(\dd x_2)
			= \sum_{n\in\mathbb{N}}w(A_n). 
		\end{align*}
		We have proved \eqref{3.2} for any family of pairwise disjoint Borel sets $(A_n)$, hence $w$ is a measure.
		Moreover $w$ is finite because 
		\begin{align*}
			\int_{I} |w|(\dd x)=\vert w \vert (I) &\leq  \iint_{I^2}  K(I, x_1, x_2) |u\otimes v|(\dd x_1\dd x_2) 
			= \iint_{I^2} |u|(\dd x_1) |v|(\dd x_2) 
			= \Vert u\Vert_{\mathcal{M}(I)}\Vert v\Vert_{\mathcal{M}(I)}.
		\end{align*}
		We have proved that 
		\begin{equation*}
			\Vert B(u, v)\Vert_{\mathcal{M}(I)} \leq  \Vert u\Vert_{\mathcal{M}(I)} \Vert v\Vert_{\mathcal{M}(I)}.
		\end{equation*}
		Hence $B(u, v)$ is a continuous bilinear map on $\mathcal{M}(X)$.
		
		\medskip 
		To prove (iii), we use Fubini's theorem in the formula  \eqref{3.1} of $B(u,v)$ as follows:
		\begin{equation*}
			\begin{array}{ll}
				\int_{I} B(u,v)  (\dd x) = B(u,v) (I)  &= \iint_{I^2} K(I, x_1, x_2)u(\dd x_1)v(\dd x_2) \\
				&= \iint_{I^2}u(\dd x_1)v(\dd x_2 )\\
				& = \int_{I}u(\dd x)\int_{I}v(\dd x),
			\end{array}
		\end{equation*}
		because $K(I, x_1, x_2) = 1$ by assumption.

		To prove (iv), we use Fubini's theorem applied to the formula  \eqref{3.1} of $B(u,v)$:
		\begin{equation*}
			\begin{array}{ll}
				\int_{I} x B(u,v)  (\dd x)  &= \iint_{I^2} \int_I x K(dx, x_1, x_2)u(\dd x_1)v(\dd x_2) \\
				&= \iint_{I^2} \dfrac{x_1+x_2}{2}u(\dd x_1)v(\dd x_2 ) \\
				&= \dfrac{\int_{I}x_1 u(\dd x_1) \int_{I}v(\dd x_2)+\int_{I}u(\dd x_1) \int_{I}x_2 v(\dd x_2) }{2},
			\end{array}
		\end{equation*}
		because $\int_I x K(dx, x_1, x_2)= \dfrac{x_1+x_2}{2}$ by Assumption \ref{ASS1.2}.
	\end{proof}

	As a consequence of Theorem \ref{TH3.1},  the map $B$ is  a bounded and bi-linear operator from $\mathcal{M}(I)\times \mathcal{M}(I)$ to $\mathcal{M}(I)$. Moreover $B$ maps $\mathcal{M}_+(I)\times \mathcal{M}_+(I)$ into $\mathcal{M}_+(I)$. To investigate the Lipschitz property of $T:\mathcal{M}_+(I) \to \mathcal{M}_+(I)$,  it is sufficient to observe that (here for short we replace $ \Vert .\Vert_{\mathcal{M}(I)}$ by $\Vert .\Vert$) 
	\begin{equation*}
		\begin{array}{ll}
			\Vert T(u)-T(v) \Vert\\
			= \bigg \Vert \Vert u \Vert^{-1}B(u,u-v)+  \left( \Vert u \Vert^{-1}-\Vert v \Vert^{-1}\right) B(u,v)+  \Vert v\Vert^{-1} B(u-v,v) \bigg \Vert \vspace{0.2cm}\\
			\leq  \bigg \Vert \Vert u \Vert^{-1} B(u,\vert u-v \vert )+  \left( \Vert u \Vert^{-1}-\Vert v \Vert^{-1} \right) B(u,v)+  \Vert v\Vert^{-1}  B(\vert u-v \vert ,v) \bigg \Vert \vspace{0.2cm}\\
			\leq 2 \, \Vert u-v \Vert + \vert \Vert u \Vert^{-1} -\Vert v \Vert^{-1}  \vert \,   \Vert u \Vert  \, \Vert v \Vert \vspace{0.2cm}\\
			\leq 2 \, \Vert u-v \Vert + \vert \Vert u \Vert-\Vert v \Vert\vert  \vspace{0.2cm}\\
			\leq 3 \, \Vert u-v \Vert,
		\end{array}
	\end{equation*}
	therefore we obtain the following proposition. 
	
	\begin{proposition} \label{PR3.2} 	Let Assumption \ref{ASS1.1} be satisfied.  The operator $T$ map $\mathcal{M}_+(I)$ into itself, and  $T$ satisfies the following properties  
		\begin{itemize}
			\item[{\rm (i)}]  $T:\mathcal{M}_+(I) \to \mathcal{M}_+(I) $ is Lipchitz continuous.
			\item[{\rm (ii)}]   $T$ is positively homogeneous. That is, 
			\begin{equation*}
				T(\lambda u)=\lambda T(u), \forall \lambda \geq 0, \forall u \in \mathcal{M}_+(I).
			\end{equation*} 
			\item[{\rm (iii)}] $T$ preserves the total mass of individuals. That is,  
			$$
			\int_{I} T(u)(dx)=	\int_{I} u(dx), \forall  u \in \mathcal{M}_+(I). 
			$$
		\end{itemize}
		If moreover Assumption \ref{ASS1.2} holds, then
		\begin{itemize}
			\item[{\rm (iv)}] $T$ preserves the total mass of transferable quantity. That is,  
			$$
			\int_{I}x T(u)(dx)=	\int_{I}x u(dx), \forall  u \in \mathcal{M}_+(I). 
			$$
		\end{itemize}
	\end{proposition}
	Therefore we obtain the following theorem. 
	\begin{theorem} \label{TH3.3}
		Let Assumption \ref{ASS1.1} be satisfied.  Then the Cauchy problem
		\begin{equation} \label{3.3}
			\partial_t u(t, dx) =2 \tau \, T\big(  u(t)\big)(dx) -  2 \tau \, u(t, dx), 
		\end{equation}
		with 
		\begin{equation} \label{3.4}
			u(0,dx)=\phi(dx) \in \mathcal{M}_+ (I),  
		\end{equation}
		generates a unique continuous homogeneous semiflow $t \to S(t)\phi $ on $\mathcal{M}_+ (I)$. That is 
		\begin{itemize}
			%	\item[{\rm (i)}]\textbf{(Semiflow property)} $S(0)\phi =\phi , \forall \phi \in \mathcal{M}_+(I)$; 
			\item[{\rm (i)}]\textbf{(Semiflow property)}  
			$$
			S(0)\phi =\phi  \text{ and }	S(t)S(s)\phi =S(t+s)\phi ,\forall t,s \geq 0, \forall \phi \in \mathcal{M}_+(I).
			$$ 
			\item[{\rm (ii)}]\textbf{(Continuity)}  The map $(t,\phi ) \to S(t) \phi $ is a continuous map from $ [0, +\infty) \times \mathcal{M}_+(I)$ to $ \mathcal{M}_+(I)$. 
			\item[{\rm (iii)}]\textbf{(Homogeneity)}  
			$$
			S(t) \lambda\phi =  \lambda S(t)\phi ,\forall t \geq 0,  \forall \lambda\geq 0, \forall \phi \in \mathcal{M}_+(I).
			$$ 
			\item[{\rm (iv)}]\textbf{(Preservation of the total mass of individuals)}  The total mass of individuals is preserved 
			$$
			\int_I	S(t)(\phi)  (dx)=  \int_I \phi(dx) ,\forall t \geq 0,  \forall \lambda\geq 0, \forall \phi \in \mathcal{M}_+(I).
			$$ 
			\item[{\rm (v)}]\textbf{(From transfer rate $1/2$ to any transfer rate $\tau>0$)}   If we define $S^\star(t)$ the semi-flow generated by \eqref{3.3}-\eqref{3.4} whenever $\tau=1/2$, then 
			$$
			S(t)=S^\star\left(2 \tau t\right), \forall t \geq 0.
			$$
		\end{itemize}
		If moreover Assumption \ref{ASS1.2} holds, then
		\begin{itemize}
			\item[{\rm (vi)}]\textbf{(Preservation of the total mass of transferable quantity)}  The total mass of transferable quantity is preserved 
			$$
			\int_I	xS(t)(\phi)  (dx)=  \int_I x \phi(dx) ,\forall t \geq 0,  \forall \lambda\geq 0, \forall \phi \in \mathcal{M}_+(I).
			$$ 
		\end{itemize}
	\end{theorem}
	\begin{remark}
		Let $\mathcal{F}:\mathcal{M}(I) \to \R$ be a positive  bounded linear form on $\mathcal{M}(I)$. We can consider for example 
		$$
		\mathcal{F}(u)=\int_{I} f(x)u(dx),
		$$
		where $f :I \to \R$ a bounded and positive continuous map on $I$. 
		
		\medskip 
		Then $(t,\phi) \mapsto U(t)\phi$ define on $[0, +\infty) \times \mathcal{M}_+ (I)$ by 
		$$
		U(t)u= \dfrac{S(t)\phi}{1+\int_{0}^{t} \mathcal{F}(S(\sigma)\phi) d \sigma}, 
		$$
		is the unique solution of  the Cauchy problem 
		$$
		u'(t)=2 \tau \, T\big( \vert  u(t) \vert \big)(dx) -  2 \tau \, u(t, dx)-\mathcal{F}(u(t))u(t,dx),
		$$ 
		with 
		\begin{equation*} 
			u(0,dx)=\phi(dx) \in \mathcal{M}_+(I).  
		\end{equation*}
		More detailed arguments can be found in Magal and Webb \cite{magal2000mutation}, and Magal \cite{magal2002global}. 
	\end{remark}

	\begin{remark}
		The rate of transfers $\tau(x)$  may vary in function of $x$ the transferable quantity. In that case, we obtain the following model  
		\begin{equation*}
			\partial_t u(t, x) = T\big( 2\tau(.) \, u(t,.)\big)(x) -  2\tau(x) u(t, x), \text{ for } x \in \R,
		\end{equation*}
		with 
		\begin{equation*} 
			u(0,dx)=\phi(dx) \in \mathcal{M}(I).  
		\end{equation*}
	\end{remark}
		Theorem \ref{TH3.3} (i) and (ii) is a direct consequence of the Cauchy-Lipschitz Theorem in Banach spaces. The properties (iii)-(v) are readily derived from the properties of $T$ and the change of variables formula for Riemann integrals. We still need to prove the preservation of first moment in (vi), which will be a consequence of the Proposition \ref{prop:further-regularity} below.

	\begin{proposition}[Improved  regularity of the semiflow]\label{prop:further-regularity}
		Let $I\subset \mathbb{R}$ be an interval and $p\geq 1$ and $u_0\in \mathcal{M}_p(I)$be given. Assume that there exists a constant $C>0$ such that  
		\begin{equation*}
			\int |x|^pK(\dd x, x_1, x_2) \leq C\big(1+|x_1|^p+|x_2|^p\big), \text{ for all } x_1, x_2\in I. 
		\end{equation*}
			Then  $\int |x|^p S(t)u_0(\dd x)<+\infty$ for any $t>0$. More precisely, the orbit $t\mapsto S(t)u_0$ is continuous on $\mathcal{M}_p(I)$ for the norm $\Vert u\Vert_{\mathcal{M}_p}= \int (1+|x|^p) |u|(\dd x)$.
	\end{proposition}
	We first prove the following Lemma. 
	\begin{lemma}\label{lem:regularity-subspace}
		Let $I\subset \mathbb{R}$ be an interval and $(X, \Vert \cdot\Vert_X)$ be a Banach space that is continuously embedded in $\mathcal{M}(I)$. Assume that  there is a constant $C>0$ such that  $\Vert B(u, u)\Vert_X\leq C\Vert u\Vert \Vert u\Vert_X$. Then $X$ is positively invariant for $S(t)$ and $t\mapsto S(t)u_0$ is continuous in $X$ for any $u_0\in X$.
	\end{lemma}
	\begin{proof}
		Let us assume without loss of generality that $\Vert u\Vert\leq \Vert u\Vert_X$ for all $u\in X$. Clearly, by the polarization identity $B(u, v)=\frac{1}{4}\big(B(u+v,u+v)-B(u-v, u-v)\big)$, it follows from our assumptions that $B$ is continuous on $X$.

		We first show that $T$ is locally Lipschitz continuous in $X$. We have:
		\begin{align*}
			\Vert T(u)-T(v)\Vert_X&=\left\Vert\dfrac{B(u, u)}{\Vert u\Vert} - \dfrac{B(v, v)}{\Vert v\Vert}\right\Vert_X \\ 
			&=\left\Vert {\Vert u\Vert^{-1}}\big({B(u, u)}-B(v,v)\big)+  B(v, v) \big(\Vert u\Vert^{-1}-\Vert v\Vert^{-1}\big) \right\Vert_X \\ 
			&\leq \Vert u\Vert^{-1}\big\Vert B(u+v, u-v)\big\Vert_X +\Vert u\Vert^{-1}\Vert v\Vert^{-1} \Vert B(u, v)\Vert_X \big|\Vert u\Vert-\Vert v\Vert\big| \\ 
			&\leq C\Vert B\Vert_X\big(\Vert u\Vert^{-1} \Vert u+v\Vert_X + \Vert u\Vert ^{-1}\Vert v\Vert ^{-1} \Vert u\Vert_X \Vert v\Vert_X\big) \Vert u-v\Vert_X ,
		\end{align*}
		so $T$ is locally Lipschitz continuous on any open set of the form $\{ u\,:\, \epsilon< \Vert u\Vert\}$ with  $\epsilon >0$. By the existence and uniqueness of the solution (recall that the norm is preserved by the semiflow $S(t)$), for any $u_0\in X$ with $\Vert u_0\Vert>0$ there exists a maximal time $T(u_0)$ such that $S(t)u_0\in X$ for all $t\in\big[0, T(u_0)\big) $ and we have the alternative $T(u_0)=+\infty$ or 
		\begin{equation*}
			\liminf_{t\to T(u_0)^-}\Vert S(t)u_0\Vert_X=+\infty.
		\end{equation*}
		Let $u_0\in \mathcal{M}(\mathbb{R})$ and suppose by contradiction that $T(u_0)<+\infty$. Let $u(t)=S(t) u_0$. Integrating by parts the equation \eqref{1.1} we get
		\begin{equation*}
			u(t) = e^{-2\tau t}u_0 + \int_0^te^{-2\tau(t-s)}T\big(u(s)\big)\dd s
		\end{equation*}
		so 
		\begin{align*}
			e^{2\tau t}\Vert u(t)\Vert_X &\leq \Vert u_0\Vert_X + \int_0^t e^{2\tau s} \Vert T\big(u(s)\big)\Vert_X\dd s \leq \Vert u_0\Vert_X + \int_0^t e^{2\tau s} C\Vert u(s)\Vert_X\dd s
		\end{align*}
		and by Gronwall's inequality we obtain 
		\begin{equation*}
			e^{2\tau t}\Vert u(t)\Vert_X \leq \Vert u_0\Vert  e^{Ct}, \text{ and finally }\Vert u(t)\Vert_X \leq e^{(C-2\tau)t} . 
		\end{equation*}
		Thus 
		\begin{equation*}
			\liminf_{t\to T(u_0)} \Vert u(t)\Vert_X \leq \Vert u_0\Vert e^{(C-2\tau)T(u_0)}<+\infty, 
		\end{equation*}
		which is a contradiction.
	\end{proof}
	We are now in the position to prove Proposition \ref{prop:further-regularity}

	\begin{proof}[Proof of Proposition \ref{prop:further-regularity}]
		We have, for $u\in\mathcal{M}_p(I)$:
		\begin{align*}
			\int(1+ |x|^p) |B(u, u)|(\dd x) &\leq \iiint (1+|x|^p)K(\dd x, x_1, x_2)|u|(\dd x_1)|u|(\dd x_2) \\ 
			&\leq \iiint 1+C\big(1+|x_1|^p+|x_2|^p\big) |u|(\dd x_1)|u|(\dd x_2) \\ 
			&=\int |u|\dd x\int|u|\dd x + C\int \big(1+|x_1|^p\big) |u|(\dd x_1)\int |u|(\dd x_2)  \\
			&\quad + C\int |u|(\dd x_1)\int \big( 1+|x_2|^p\big)|u|(\dd x_2) \\ 
			&=\Vert u\Vert \Vert u\Vert + C\big(\Vert u\Vert_{\mathcal{M}_p(I)}\Vert u\Vert + \Vert u\Vert\Vert u\Vert_{\mathcal{M}_p(I)}\big)\leq C'\Vert u\Vert \Vert u\Vert_{\mathcal{M}_p(I)}, 
		\end{align*}
		for some $C'>0$.
		Hence we can apply Lemma \ref{lem:regularity-subspace} with $X=\mathcal{M}_p(I)$, which proves Proposition \ref{prop:further-regularity}.
	\end{proof}
	\begin{remark}[Robin Hood model]
		In the case of the Robin Hood model described in \eqref{eq:RH} above, we have
		\begin{equation*}
			\int |x|^pK_1(\dd x, x_1, x_2) =\frac{1}{2}\left( \big|(1-f) x_1+fx_2\big|^p + \big|f x_1+(1-f)x_2\big|^p \right) \leq |x_1|^p + |x_2|^p, 
		\end{equation*}
		so Proposition \ref{prop:further-regularity} can be applied for any $p\geq 1$.
	\end{remark}
	\begin{remark}[Sheriff of Nottingham model]
		In the case of the Sheriff of Nottingham model described in \eqref{eq:Sheriff-Nottingham} above, we have
		\begin{equation*}
			\int |x|^pK_2(\dd x, x_1, x_2) =\frac{1}{2}\left( \big|(1+f) x_1-fx_2\big|^p + \big|-f x_1+(1+f)x_2\big|^p \right) \leq 2^{p-1}(1+f)^p\big(|x_1|^p + |x_2|^p\big), 
		\end{equation*}
		so Proposition \ref{prop:further-regularity} can be applied for any $p\geq 1$.
	\end{remark}

	\section{Understanding  \eqref{1.1} in $L^1(\R)$} 
	\label{Section5}
	Recall that a Borel subset $A \in\mathcal{B}(I)$  is said to be \textbf{negligible} if and only if $A$ has a null Lebesgue measure. 
	We consider the following assumption on the kernel $K$.
	\begin{assumption}\label{ASS5.1}
	    For each negligible subset  $A\in\mathcal{B}(I)$, the set
		\begin{equation*} 
			\mathcal{N}(A):=\left\{(x_1, x_2)\in I\times I\,:\, K(A, x_1, x_2)\neq 0\right\}, 
		\end{equation*}
		is negligible. 
	\end{assumption}
We note that an equivalent way to state Assumption \ref{ASS5.1} is the following: for each negligible subset $A\in\mathcal{B}(I)$, we have
\begin{equation*}
    \iint_{I\times I} K(A, x_1, x_2)\dd x_1\dd x_2=0. 
\end{equation*}

	Thanks to the Radon-Nikodym Theorem (which is recalled in the Supplementary Materials as Theorem \ref{THA6}), we have the following characterization of kernels that define a bilinear mapping from $L^1(I)\times L^1(I)$ to $L^1(I)$.
	\begin{proposition}\label{PROP4.1}
		Let Assumption \ref{ASS1.1} be satisfied.  Then we have
		\begin{equation*} 
			B(u, v)\in L^1(I) \text{ for all } (u, v)\in L^1(I)\times L^1(I)
		\end{equation*}
		if, and only if, 	Assumption \ref{ASS5.1} is satisfied. 
	\end{proposition}
	
	\begin{proof}
		For simplicity, here we call $\mathcal{L}$ the one-dimensional Lebesgue measure, that is to say $\mathcal{L}(A)=\int_A \dd x$; and $\mathcal{L}^2:=\mathcal{L}\otimes\mathcal{L}$ the two-dimensional Lebesgue measure in $\R^2$. 
		
		\medskip 
		Let $u, v\in L^1(I)$ be given. Suppose that 
		$$
		\mathcal{L}^2(\mathcal{N}(A))=0,
		$$ for each $ A\in\mathcal{B}(I)$ with $\mathcal{L}(A)=0$. Then by definition (see \eqref{3.1}),  
		\begin{equation*}
			B(u, v)(A) = \iint_{I\times I} K(A, x_1, x_2)u(x_1)\dd x_1 v(x_2)\dd x_2 =0,
		\end{equation*}
		since $(x_1, x_2)\mapsto K(A, x_1, x_2)$ is equal to zero $\mathcal{L}^2$-almost everywhere in $\R^2$ by assumption. Therefore $B(u, v)$ is absolutely continuous with respect to  the Lebesgue measure $\mathcal{L}$, and by  the Radon-Nikodym Theorem \ref{THA6}, we can find function $f\in L^1(I)$ such that 
		\begin{equation*}
			B(u, v)(dx)= f(x)\dd x,
		\end{equation*}
		which is equivalent to 
		\begin{equation*}
			B(u, v)\in L^1(I). 
		\end{equation*}
		Conversely, assume that $B(u, v)\in L^1(I)$ for any $(u, v)\in L^1(I)^2$. If $I$ is bounded then $1\in L^1(I)$, so taking $u=v=1$, and $B(1, 1)(\dd x) = f(x)\dd x$ with $f\in L^1(I)$  gives 
		\begin{equation*}
			B(u, v)(A) = \iint_{I\times I} K(A, x_1, x_2)\dd x_1\dd x_2=\int_{A} f(x)\dd x = 0, 
		\end{equation*}
		whenever $\mathcal{L}(A)=0$, and we are done. 
		
		Let us consider the case when $I$ is not bounded. Assume that  $A\in \mathcal{B}(I)$  is negligible. Define  
		$$
		u_n(x)=v_n(x)=\mathbbm{1}_{[-n, n]\cap I}(x)
		$$ 
		where $x \mapsto \mathbbm{1}_{E}(x)$ is the indicator function of the set $E$. 
		
		Then,	we have by assumption, 
		$$
		B(u_n, v_n)=f_n(x)\dd x
		$$
		for some $f_n\in L^1(I)$. 
		
		\medskip 
		\noindent 	Moreover 
		\begin{equation*}
			B\big(u_n, v_n\big)(A) = \iint_{(I\cap [-n, n])\times (I\cap [-n,n])} K(A, x_1, x_2)\dd x_1\dd x_2=\int_{A} f_n(x)\dd x = 0,
		\end{equation*}
		thus
		\begin{equation*}
			\mathcal{L}\big(\mathcal{N}(A)\cap [-n, n]^2\big) = 0 \text{ for all }n\in\mathbb{N}. 
		\end{equation*}
		Finally since we have an increase sequence of subsets, we obtain 
		\begin{equation*}
			\mathcal{L}\big(\mathcal{N}(A)\big) = \mathcal{L}\left(\mathcal{N}(A)\cap \bigcup_{n\in\mathbb{N}}[-n, n]^2\right) = \lim_{n\to+\infty} \mathcal{L}\left(\mathcal{N}(A)\cap [-n, n]^2\right)=0,
		\end{equation*}
		and	the proof is completed. 		
	\end{proof}
	Since the norm in the space of measure coincides with the $L^1$ norm for an $L^1$ function, we deduce that $T$ maps $L^1_+(I)$ into $L^1_+(I)$  into itself, and the following  statements are consequences of Theorem \ref{TH3.3}, and Proposition \ref{PROP4.1}.  
	\begin{theorem}	\label{thm:wp-L1}
	    Let Assumption \ref{ASS1.1} and \ref{ASS5.1} be satisfied and consider the Cauchy problem
	    \begin{equation} \label{4.1}
			\partial_t u(t, x) =2 \tau \, T\big(  u(t)\big) -  2 \tau \, u(t, x), 
		\end{equation}
		with 
		\begin{equation} \label{4.2}
			u(0,x)=\phi(x) \in L^1_+ (I).  
		\end{equation}
		The Cauchy problem \eqref{4.1}-\eqref{4.2} generates a unique semiflow which is the restriction of $S(t)$ to $L^1_+(I)$. We deduce that   
		$$
		S(t)L^1_+(I) \subset L^1_+(I), \forall t \geq 0,  
		$$
		and the semiflow $t \to S(t)\phi $ restricted to $L^1_+ (I)$ satisfies the following properties: 
		\begin{itemize}
			\item[{\rm (i)}]\textbf{(Continuity)}  The map $(t,\phi ) \to S(t) \phi $ is a continuous map from $ [0, +\infty) \times L^1_+(I)$ to $ L^1_+(I)$. 
			\item[{\rm (ii)}]\textbf{(Preservation of the total mass of individuals)}  The total mass of individuals is preserved 
			$$
			\int_I	S(t)(\phi)  (x)dx=  \int_I \phi(x)dx ,\forall t \geq 0,  \forall \lambda\geq 0, \forall \phi \in \mathcal{M}_+(I).
			$$ 
		\end{itemize}
		If moreover Assumption \ref{ASS1.2} holds, then we have in addition
		\begin{itemize}
			\item[{\rm (iii)}]\textbf{(Preservation of the total mass of transferable quantity)}  The total mass of transferable quantity is preserved 
			$$
			\int_I	xS(t)(\phi)  (x)dx=  \int_I x \phi(x)dx ,\forall t \geq 0,  \forall \lambda\geq 0, \forall \phi \in \mathcal{M}_+(I).
			$$ 
		\end{itemize}
	\end{theorem}
	\begin{example}[Robin Hood model]
		
		Let $K(\dd x, x_1, x_2)=K_1(\dd x, x_1, x_2)=\frac{1}{2}\big(\delta_{x_2-f(x_2-x_1)}(\dd x)+\delta_{x_1-f(x_1-x_2)}(\dd x)\big)$. If $A\in \mathcal{B}(I)$ has zero Lebesgue measure, then we have:
		\begin{align*} 
			\mathcal{N}(A)&=\{(x_1, x_2)\in I\times I\,:\, K(A, x_1, x_2)>0\} \\ 
			&=\{(x_1, x_2)\,:\, x_2-f(x_2-x_1)=y \in A \text{  or } x_1-f(x_1-x_2)=z\in A\}\\
			&=\left\{(x_1, x_2)\,:\,x_1=\frac{1-f}{1-2f}z-\frac{f}{1-2f}y  \text{ and }  \right. \\
			& \hspace{1cm}\left. x_2=\frac{1-f}{1-2f}y-\frac{f}{1-2f}z \text{ and } \big(y\in A \text{ or } z\in A\big)\right\}\\
			&=\left\{\left(\frac{1-f}{1-2f}z-\frac{f}{1-2f}y, \frac{1-f}{1-2f}y-\frac{f}{1-2f}z\right)\,:\, y\in A, z\in I\right\} \\
			&\quad \cup \left\{\left(\frac{1-f}{1-2f}z-\frac{f}{1-2f}y, \frac{1-f}{1-2f}y-\frac{f}{1-2f}z\right)\,:\, y\in I, z\in A\right\}.
		\end{align*}
		The two sets above have zero Lebesgue measure because they are the image of $A\times I$ and $I\times A$ by a linear invertible transformation. Therefore $\mathcal{N}(A)$ has zero Lebesgue measure and we can apply Proposition \ref{PROP4.1}.
		
	\end{example}
	\begin{example}[Sheriff of Nottingham model]
		
		Let $K(\dd x, x_1, x_2)=K_2(\dd x, x_1, x_2)=\frac{1}{2}\big(\delta_{x_2+f(x_2-x_1)}(\dd x)+\delta_{x_1+f(x_1-x_2)}(\dd x)\big)$. If $A\in \mathcal{B}(\mathbb{R})$ has zero Lebesgue measure, then we have:
		\begin{align*} 
			\mathcal{N}(A)&=\{(x_1, x_2)\in I\times I\,:\, K(A, x_1, x_2)>0\} \\ 
			&=\{(x_1, x_2)\,:\, x_2+f(x_2-x_1)=y \in A \text{  or } x_1+f(x_1-x_2)=z\in A\}\\
			&=\left\{\left(\frac{1+f}{1+2f}z+\frac{f}{1+2f}y, \frac{1+f}{1+2f}y+\frac{f}{1+2f}z\right)\,:\, y\in A, z\in I\right\} \\
			&\quad \cup \left\{\left(\frac{1+f}{1+2f}z+\frac{f}{1+2f}y, \frac{1+f}{1+2f}y+\frac{f}{1+2f}z\right)\,:\, y\in I, z\in A\right\}.
		\end{align*}
		The two sets above have zero Lebesgue measure because they are the image of $A\times \mathbb{R}$ and $\mathbb{R}\times A$ by a linear invertible transformation. Therefore $\mathcal{N}(A)$ has zero Lebesgue measure and we can apply Proposition \ref{PROP4.1}.
		
	\end{example}
	Similarly, the mixed Robin Hood and Sheriff of Nottingham model also define a bilinear mapping from $L^1\times L^1$ to $L^1$.
	\begin{example}[Distributed Robin Hood or Sheriff of Nottingham models]	\label{ex:Distributed-RH}
		The kernel of the distributed Robin Hood  model consists in replacing the Dirac mass centered a $0$, by $x \to g(x)\in L^1(I)$ a density of probability centered at $0$. That is 
		\begin{equation}\label{eq:distributed-RH}
			K_3(\dd x, x_1, x_2) =\frac{1}{2}  \bigg\{ g\left( x- \left[ x_2 - f( x_2-x_1) \right] \right)+ g\left( x- \left[ x_1 - f(x_1-x_2) \right] \right) \bigg\} \, \d x.
		\end{equation}
		Similarly, the kernel of the distributed Sheriff of Nottingham model is the following 
		\begin{equation*}
			K_4(\dd x, x_1, x_2) :=\frac{1}{2} \bigg\{ g\left( x- \left[ x_2 + f( x_2-x_1) \right] \right)+ g\left( x- \left[ x_1 + f(x_1-x_2) \right] \right) \bigg\}  \, \d x. 
		\end{equation*}
		Let $K(\dd x, x_1, x_2)=K(x, x_1, x_2)\dd x$ with $K(x, x_1, x_2)\in L^1(I)$ for any $(x_1, x_2)\in I\times I$. Examples are the distributed Robin Hood model, distributed Sheriff of Nottingham model, and distributed mixed Robin Hood and Sheriff of Nottingham model.  If $A\in \mathcal{B}(I)$ has zero Lebesgue measure, then we have automatically
		\begin{align*} 
			K(A, x_1, x_2) = \int_{A} K(x, x_1, x_2)\dd x = 0 \text{ for any } (x_1, x_2)\in I\times I,  \text{ so } \mathcal{N}(A)=\varnothing.
		\end{align*}
		Therefore we can apply Proposition \ref{PROP4.1}.
		
	\end{example}
	
	\section{Asymptotic behavior}
	In this section we prove some qualitative results about the asymptotic behavior of the (distributed) Robin Hood and Sheriff of Nottingham models. 

	\subsection{Robin Hood model}
	\label{sec:asymptotic-robin-hood}
	In the case of the Robin Hood model, we can describe the asymptotic behavior of the solutions starting from an initial measure with finite second moment, thanks to the explicit dynamics of the variance. The following formula was remarked in Bisi \cite{Bisi-2016}, but we recall here the computations for completeness. Set $\tau=\frac{1}{2}$ and $\int_I u_0(\dd x)=1$ for simplicity. 	Suppose that $u_0(\dd x)\in \mathcal{P}_2(I)$ with $I$ bounded or not, then the solution $u(t, \dd x)$ of \eqref{1.1} with $u(0, \dd x)=u_0(\dd x)$ belongs to $\mathcal{P}_2(I)$ by Proposition \ref{prop:further-regularity}, so in particular $M_2\big(u(t, \dd x)\big) <+\infty$ and $V(u)<+\infty$ for all $t>0$. We have
	\begin{align*}
	    \dfrac{\dd}{\dd t} V(u) &= \int_I\int_I\int_I \big(x-M_1(u)\big)^2 K_1(\dd x, x_1, x_2) u(\dd x_1) u(\dd x_2) - \int_I \big(x-M_1(u)\big)^2 u(\dd x)  \\
	    &= \int_I\int_I\int_I \big(x-M_1(u)\big)^2 \delta_{(1-f)x_1+fx_2}(\dd x) u(\dd x_1) u(\dd x_2)- V(u) \\
	    &= \int_I\int_I \big((1-f)x_1+fx_2-M_1(u)\big)^2 u(\dd x_1) u(\dd x_2)- V(u) \\
	    &= \int_I\int_I (1-f)^2 x_1^2 + f^2 x_2^2 + M_1(u)^2  + 2f(1-f)x_1 x_2 \\ 
	    &\quad\qquad  - 2M_1(u)\big((1-f)x_1+fx_2)\big)u(\dd x_1) u(\dd x_2)- V(u) \\
	    &=(1-f)^2 M_2(u)+f^2 M_2(u) +M_1(u)^2+ 2f(1-f) M_1(u)^2 -2M_1(u)^2 - V(u) \\
	    &=\big[(1-f)^2 +f^2\big] \big(M_2(u)-M_1(u)^2\big)+\big[(1-f)^2 +f^2 + 2f(1-f)-1\big]M_1(u)^2 - V(u)  \\
	    &=\big[(1-f)^2 +f^2-1\big] V(u)=-2f(1-f) V(u),
	\end{align*}
	so the variance of $u$ converges exponentially fast towards $0$. Since moreover $M_1(u) = \int_I xu(\dd x)$ is a constant, we obtain the following result. 
	\begin{proposition}[Convergence of the Robin Hood model]
		Let $I$ be an  interval of $\mathbb{R}$ and $u_0\in \mathcal{P}_2(I)$. Then the solution $u(t, \dd x)$ to $\eqref{1.1}$ with $u(0, \dd x)= u_0(\dd x)$ and $K=K_1$ as in \eqref{eq:RH}, converges in the sense of the weak-$\star$ topology towards the Dirac mass centered at $M_1(u_0)$, 
	    \begin{equation*}
		u(t, \dd x)\xrightarrow[t\to+\infty]{\text{weak}-\star}u_\infty(\dd x)=\delta_{M_1(u_0)}(\dd x).
	    \end{equation*}
	    In other words, for each $\varphi\in BC(I)$ we have
	    \begin{equation*}
		\int_I \varphi(x) u(t, \dd x) \xrightarrow[t\to+\infty]{} \varphi\big(M_1(u_0)\big).
	    \end{equation*}
	\end{proposition}
	Note that it is hopeless to obtain a strong convergence for all $u_0\in \mathcal{M}_+(I)$; indeed $u(t, \dd x)\in L^1(\mathbb{R})$ if $u_0(\dd x)\in L^1(\mathbb{R})$ by Theorem \ref{thm:wp-L1}.

	\subsection{Distributed Robin Hood model}
	\label{sec:asymptotic-distributed-robin-hood}
	Recall the distributed Robin Hood model $K_3$ defined in \eqref{eq:distributed-RH}. Here we assume that $g\in L^1_+(\mathbb{R})$ is a probability density centered at $0$ (i.e. $\int g(x)\dd x=1$, $\int_{\mathbb{R}} xg(x)\dd x=0$). We also assume that  $x^p g(x)\in L^1(\mathbb{R})$ for some $p\geq 1$. We have
	\begin{align*}
		\int |x|^p K_3(\dd x, x_1, x_2) & = \int \frac{1}{2}\left( \big|x+(1-f)x_1+f x_2\big|^p+ \big|x+fx_1+(1-f) x_2\big|^p\right)g(x)\dd x  \\ 
		&\leq \int \frac{1}{2}\left( \big|(1-f)(x+x_1)|^p +|f(x+ x_2)|^p+ \big|f(x+x_1)|^p +|(1-f)(x+ x_2)|^p\right)g(x)\dd x \\
		&\leq \int \big[(1-f)^p+f^p\big]2^{p-1}\big(2|x|^p+|x_1|^p + |x_2|^p\big) g(x)\dd x \\ 
		&\leq \big[(1-f)^p+f^p\big]2^{p-1}\left( 2\int |x|^p g(x)\dd x+|x_1|^p+|x_2|^p\right),
	\end{align*}
thus we can apply  Proposition \ref{prop:further-regularity} and the $n$-th moments of $u(t, \dd x)$ are well-defined for all $t>0$ for all $n\leq p$. We can then prove the following result:
	\begin{proposition}[Convergence of the moments]\label{prop:conv-moments}
		Let $g\in L^1_+(\mathbb{R})$ be such that $|x|^pg(x)\in L^1(\mathbb{R})$, $\int g(x)\dd x=1$, and  $\int xg(x)=0$, and $f\in(0,1)$. Let $u_0\in\mathcal{P}_p(\mathbb{R})$ and $S(t)$ the semiflow given by Theorem \ref{TH3.3} with $K=K_3$ as defined in \eqref{eq:distributed-RH}. Then there exists a unique sequence of real numbers $(m_n)_{0\leq n\leq p}$, which depends only on $M_1(u_0)$, such that 
	    \begin{equation*}
		M_n(S(t)u_0)\xrightarrow[t\to+\infty]{} m_n, \text{ for all }n\in\mathbb{N} , n\leq p.
	    \end{equation*}
	    In particular if $p\geq 2$, then  $V(S(t)u_0)$ converges to a finite number:
	    \begin{equation} \label{eq:limit-variance}
		V\big(S(t)u_0\big)\xrightarrow[t\to+\infty]{} \frac{1}{2f(1-f)} V(g).
	    \end{equation}
	\end{proposition}
	\begin{proof}
	    Let us start with the case of the variance for which the computations are easier to track.  We have:
	    \begin{align*}
		\frac{\dd }{\dd t}M_2(u)&= \iiint \big(x+(1-f)x_1+fx_2\big)^2 g(x) \dd x u(\dd x_1) u(\dd x_2) - M_2(u) \\ 
		&= \int x^2 g(x) \dd x + (1-f)^2 \int x_1^2 u(\dd x_1) + f^2 \int x_2^2 u(\dd x_2) \\ 
		&\quad + 2(1-f)\int xg(x)\dd x \int x_1u(\dd x_1)+2f \int xg(x)\dd x \int x_2u(\dd x_2)\\
		&\quad + 2f(1-f) \int x_1 u(\dd x_1)\int x_2 u(\dd x_2) - M_2(u)\\
		&=M_2(g) + \big[(1-f)^2 + f^2 \big] M_2(u) + 2f(1-f) M_1(u)^2 - M_2(u) \\ 
		&=M_2(g) - 2f(1-f) M_2(u) + 2f(1-f)M_1(u)^2.
	    \end{align*}
	    Here we wrote $u$ instead of $u(t, \dd x)$ to avoid unnecessarily long lines. 
	    By solving explicitly this equation, it is not difficult to show that $M_2$ converges to a finite value, which satisfies
	    \begin{equation*}
		M_2\big(u(t, \dd x)\big)\xrightarrow[t\to+\infty]{} m_2:=\frac{1}{2f(1-f)} M_2(g)+M_1(u_0)^2. 
	    \end{equation*}
	    Hence we recover \eqref{eq:limit-variance} by remarking that $V(u)=M_2(u)-M_1(u)^2$.\medskip

	    More generally, fix $n\in\mathbb{N}$, $n\leq p$, then we have
	    \begin{align*}
		\dfrac{\dd}{\dd t} M_n(u)&= \iiint \big(x+(1-f)x_1+fx_2\big)^n g(x)\dd x u(\dd x_1) u(\dd x_2)-M_n(u) \\
		&=\sum_{k=0}^n \sum_{i=0}^k \binom{n}{k}\binom{k}{i} M_{n-k}(g) (1-f)^{k-i}f^iM_{k-i}(u)M_i(u)  - M_n(u) \\ 
		&= \big[(1-f)^n+f^n-1\big] M_n(u) + \underset{(i, k)\not\in\{(0, n), (n,n)\}}{\sum_{1\leq i\leq k\leq n}} \binom{n}{k}\binom{k}{i} M_{n-k}(g) (1-f)^{k-i}f^iM_{k-i}(u)M_i(u) 
	    \end{align*}
	    Clearly $(1-f)^n+f^n-1<0$,  thus by an immediate recursion, we can prove that $M_n(u)$ converges to a number $m_n$ satisfying
	    \begin{equation*}
		m_n=\frac{1}{1-(1-f)^n-f^n}\underset{(i, k)\not\in\{(0, n), (n,n)\}}{\sum_{1\leq i\leq k\leq n}} \binom{n}{k}\binom{k}{i} M_{n-k}(g) (1-f)^{k-i}f^im_{k-i}m_i
	    \end{equation*}
	    Since $m_2$ depends only on $M_1(u_0)$, we obtain by induction that $m_n$ depends only on $M_1(u_0)$, thanks to the recursion formula. This finishes the proof of Proposition \ref{prop:conv-moments}.
	\end{proof}

	Actually, we can go a bit further and prove the weak convergence to a unique stationary distribution thanks to an argument inspired by Matthes and Toscani \cite{Matthes-Toscani-2008} and Pareschi and Toscani \cite{Pareschi-Toscani-2006}. We first define the Fourier transform of measures, 
	\begin{equation}
		\mathcal{F}[u]:= \int_{\mathbb{R}} e^{ix\xi} u(\dd x), \quad u\in\mathcal{M}(\mathbb{R}),
	\end{equation}
	where $i^2=-1$ is the complex imaginary unit. We introduce the space of probability measures with finite second moment and fixed first moment, 
	\begin{equation*}
		X_{M}:=\left\{u\in \mathcal{P}_{2}(\mathbb{R})\,:\, \int_\mathbb{R} xu(\dd x) = M\right\}, 
	\end{equation*}
	where $M\in\mathbb{R}$.
		Then for $s\in(0, 1)$ we introduce the $(1+s)$-Fourier distance
	\begin{equation}
		d_s(u, v):= \sup_{\xi\in\mathbb{R}\backslash\{0\}} \dfrac{|\mathcal{F}[u](\xi)-\mathcal{F}[v](\xi)|}{|\xi|^{1+s}}, u, v\in \mathcal{M}(\mathbb{R}).
	\end{equation}
	It is classical (see Carrillo and Toscani \cite[Proposition 2.7]{Carrillo-Toscani-2007}) that $d_s$ is a distance on $X_M$ and that the metric space $(X_M, d_s)$ is complete, for any $M\in\mathbb{R}$ and $s\in(0,1)$. Note that convergence in Fourier distance implies weak convergence in the sense of measures.

 We first prove the existence and uniqueness of a stationary solution in $X_M$.
	\begin{lemma}[Contractivity of the transfer operator] \label{lem:contractive-transfer}
		Let $M\in\mathbb{R}$, then under the assumptions of Proposition \ref{prop:conv-moments},  the operator $T$ leaves $X_M$ invariant and is a contraction on $X_M$. More precisely, 
		\begin{equation}\label{eq:contractive-transfer}
			d_s\big(T(u), T(v)\big)\leq \big[(1-f)^{1+s}+f^{1+s}\big] d_s(u, v), \qquad \,^\forall u,v \in X_M.
		\end{equation}
	\end{lemma}
	\begin{proof}

		We first prove the invariance of $X_M$.   We have, by the same computations as in the proof of Proposition \ref{prop:conv-moments}
:
	    \begin{align*}
		M_2\big(T(u)\big)&=\int x^2 T(u)(\dd x)= \iiint \big(x+(1-f)x_1+fx_2\big)^2 g(x) \dd x u(\dd x_1) u(\dd x_2) \\ 
		&=M_2(g) + \big[(1-f)^2 + f^2 \big] M_2(u) + 2f(1-f) M_1(u)^2 <+\infty.
	    \end{align*}
	    Then by Proposition \ref{PR3.2}, the first moment is preserved by $T$: $\int xT(u)\dd x = \int xu(\dd x)$. We have proved the stability of $X_M$.

		Next we prove the contractivity of $T$ for the distance $d_s$. We have
		\begin{align*}
		    \mathcal{F}[T(u)](\xi) &= \iiint e^{ix\xi}g\big(x-(1-f)x_1-fx_2\big)\dd xu(\dd x_1)u(\dd x_2) \\
			&= \int e^{ix\xi}g(x)\dd x\int e^{i(1-f)x_1 \xi}u(\dd x_1) \int e^{ifx_2} u(\dd x_2) = \mathcal{F}[g](\xi)\times \mathcal{F}[u]\big((1-f)\xi\big) \times \mathcal{F}[u]\big(f\xi\big),
		\end{align*}
		and we deduce that for any $\xi\in\mathbb{R}\backslash\{0\}$
		\begin{align*}
		    \dfrac{|\mathcal{F}[T(u)](\xi)-\mathcal{F}[T(v)](\xi)|}{|\xi|^{1+s}}&=\dfrac{\left| \mathcal{F}[g](\xi)\mathcal{F}[u]\big((1-f)\xi\big) \mathcal{F}[u]\big(f\xi\big) - \mathcal{F}[g](\xi)\mathcal{F}[v]\big((1-f)\xi\big) \mathcal{F}[v]\big(f\xi\big)\right|}{|\xi|^{1+s}} \\
			&\leq |\mathcal{F}[g](\xi)| \left( |\mathcal{F}[u]\big((1-f) \xi\big)|\dfrac{|\mathcal{F}[u](f\xi)-\mathcal{F}[v](f\xi)|}{|\xi|^{1+s}} \right. \\
			&\quad \left.+ |\mathcal{F}[v](f \xi)|\dfrac{|\mathcal{F}[u]\big((1-f)\xi\big)-\mathcal{F}[v]\big((1-f)\xi\big)|}{|\xi|^{1+s}}\right) \\ 
			&\leq (1-f)^{1+s} d_s(u, v)+f^{1+s} d_s(u,v)  = \big((1-f)^{1+s}+f^{1+s}\big) d_s(u, v).
		\end{align*}
		Taking the supremum over all $\xi\in\mathbb{R}\backslash \{0\}$, we obtain \eqref{eq:contractive-transfer}. The proof of Lemma \ref{lem:contractive-transfer} is completed.
	\end{proof}
	As an immediate consequence of Lemma \ref{lem:contractive-transfer} and the Banach fixed-point Theorem (certainly $(1-f)^{1+s}+f^{1+s}<1$ by the strict concavity of $f\mapsto f^{1+s}$), for each $M\in\mathbb{R}$ there exists a unique stationary distribution $u^\infty_M(\dd x)\in X_M$ such that 
	\begin{equation}\label{eq:stationary}
		T[u^\infty_M] =u^\infty_M.
	\end{equation}
	We deduce the following result.
	\begin{theorem}\label{eq:DistributedRH-stability}
		Let $p\geq 2$ and $f\in(0,1)$ be given and  $g\in L^1_+(\mathbb{R})$ be such that $|x|^pg(x)\in L^1(\mathbb{R})$, $\int g(x)\dd x=1$, and  $\int xg(x)=0$. 
		Let $u_0, v_0\in \mathcal{P}_2(\mathbb{R})$ be such that $\int xu_0(\dd x) = \int xv_0(\dd x)$. Then we have
		\begin{equation}\label{eq:contraction-orbits}
			d_s\big(S(t) u_0, S(t) v_0\big)\leq d_s(u_0, v_0) e^{-(1-(1-f)^{1+s}-f^{1+s}) t}, \qquad  \,^\forall t>0. 
		\end{equation}

		In particular 
		\begin{equation*}
			d_s(S(t)u_0, u^\infty_M)\xrightarrow[t\to+\infty]{}0, 
		\end{equation*}
		where $M=\int xu_0(\dd x)$ and $u^\infty_M$ is the unique solution of \eqref{eq:stationary} in $X_M$. 
	\end{theorem}
\begin{proof}
    Let us first prove \eqref{eq:contraction-orbits}. For notational simplicity, let $u(t, \dd x):=S(t)u_0(\dd x) $ and $v(t, \dd x) = S(t) v_0(\dd x)$; let us define moreover 
    \begin{equation*}
	d(t, \xi):= \frac{1}{|\xi|^{1+s}} \left(\mathcal{F}[u(t,\cdot)](\xi) - \mathcal{F}[u(t,\cdot)](\xi)\right).
    \end{equation*}
    Clearly $d(t, \xi)$ is continuously differentiable in time for any fixed $\xi\in\mathbb{R}\backslash\{0\}$ and we have
    \begin{align*}
	\frac{\partial}{\partial t}d(t, \xi) &= \frac{1}{|\xi|^s} \big(\mathcal{F}[T(u)(t, \cdot)](\xi)  - \mathcal{F}[u(t, \cdot)](\xi) - \mathcal{F}[T(v)(t, \cdot)](\xi)  + \mathcal{F}[v(t, \cdot)](\xi)\big),  \\ 
	&=\frac{1}{|\xi|^s} \big(\mathcal{F}[T(u)(t, \cdot)](\xi)  - \mathcal{F}[T(v)(t, \cdot)](\xi) \big) - d(t, \xi).
    \end{align*}
    Using the integration by parts formula in the above equation, we get 
    \begin{align*}
	d(t, \xi) = e^{-t} d(0, \xi) + \int_0^t e^{-(t-\sigma)} \frac{1}{|\xi|^s}\big(\mathcal{F}[T(u)(\sigma, \cdot)](\xi)  - \mathcal{F}[T(v)(\sigma, \cdot)](\xi) \big)\dd \sigma,
    \end{align*}
    thus by using \eqref{eq:contractive-transfer} and the triangle inequality we obtain
    \begin{align*}
	|d(t, \xi)|e^t &\leq |d(0, \xi)|+\int_0^t e^\sigma\frac{1}{|\xi|^s} \big|\mathcal{F}[T(u)(\sigma, \cdot)](\xi)  - \mathcal{F}[T(v)(\sigma, \cdot)](\xi) \big|\dd \sigma \\ 
	&\leq d_s(u_0, v_0)+\int_0^t\big[(1-f)^{1+s}+f^{1+s}\big] e^\sigma d_s\big(u(\sigma, \cdot), v(\sigma, \cdot)\big)\dd \sigma.
    \end{align*}
    By taking the supremum over all $\xi\in \mathbb{R}\backslash\{0\}$ we obtain
    \begin{equation*}
	d_s\big(u(t, \cdot), v(t, \cdot)\big)e^t\leq d_s(u_0, v_0)+\int_0^t \big[(1-f)^{1+s}+f^{1+s}\big]e^\sigma d_s\big(u(\sigma, \cdot), v(\sigma, \cdot)\big)\dd \sigma,
    \end{equation*}
    so by applying the integrated form of Gronwall's inequality 
    \begin{equation*}
	d_s\big(u(t, \cdot), v(t, \cdot)\big)e^t\leq d_s(u_0, v_0)e^{[(1-f)^{1+s}+f^{1+s}]t} .
    \end{equation*}
    This is exactly \eqref{eq:contraction-orbits}, which finishes the proof of the first part of the Theorem.

    Next, since $T(u^\infty_M) = u^\infty_M$ we have that $S(t) u^\infty_M = u^\infty_M \in \mathcal{P}_2(\mathbb{R})$ so we can apply the first part  of the Theorem to find that 
		\begin{equation*}
		    d_s\big(S(t) u_0, u^\infty_M)=d_s\big(S(t) u_0, S(t) u^\infty_M)\leq d_s(u_0, u^\infty_M) e^{-(1-(1-f)^{1+s}-f^{1+s}) t}\xrightarrow[t\to+\infty]{}0. 
		\end{equation*}
		This finishes the proof of Theorem \ref{eq:DistributedRH-stability}.
    \end{proof}
    \subsection{Sheriff of Nottingham model}
    In the case of the Sheriff of Nottingham model $K=K_2$ define in \eqref{eq:Sheriff-Nottingham} in section \ref{sec:Sheriff-Nottingham}, we do not expect a convergence of the moments but an explosion. Indeed, let $u_0\in\mathcal{P}_2(\mathbb{R})$ be given. Then reproducing the computations in section \ref{sec:asymptotic-robin-hood} leads to 
    \begin{equation*}
	\frac{\dd}{\dd t} V(u) = \big[(1+f)^2 + f^2 -1\big]V(u) = \epsilon V(u),  
    \end{equation*}
    where $u(t, \dd x)=S(t)u_0$ and  $\epsilon:=(1+f)^2 + f^2 -1$ is a positive constant. Thus 
    \begin{equation*}
	V(S(t)u_0)= V(0)e^{\epsilon t}\xrightarrow[t\to+\infty]{}+\infty.
    \end{equation*}
    Summarizing, we have the following result.
    \begin{proposition}[Sheriff of Nottingham: divergence of the variance]
	Let $u_0\in\mathcal{P}_2(\mathbb{R})$ satisfy $\int_{I} u_0(\dd x)=1$. Then the variance of $S(t)u_0$ explodes as $t\to+\infty$: 
    \begin{equation*}
	V(S(t)u_0)\xrightarrow[t\to+\infty]{}+\infty.
    \end{equation*}
    \end{proposition}
	In the case of the distributed Sheriff of Nottingham model, by reproducing the computations in the proof of Proposition \ref{prop:conv-moments} we obtain, assuming that the second moment of $u(t, \dd x)=S(t)u_0$ stays finite at all times,  
	\begin{equation*}
	    \frac{\dd}{\dd t} M_2(u) = M_2(g)-2f(1+f)M_1(u)^2 + \big[(1+f)^2+f^2-1\big]M_2(u)
	\end{equation*}
	so we obtain
	\begin{equation*}
	    M_2(u(t))\xrightarrow[t\to+\infty]{}+\infty.
	\end{equation*}
	Here again, the second moment (and hence the variance) cannot be uniformly bounded for all times. 

	\section{Numerical simulation}
	\label{Section6}
	We introduce $p \in [0,1]$, the population's redistribution fraction. The parameter $p$ is also the probability of applying the Robin Hood (RH) model during a transfer between two individuals. Otherwise, we use the Sheriff of Nottingham (SN) model with the probability $1-p$. In that case, the model is the following
	
	\begin{equation} \label{5.1}
		\partial_t u(t, dx) =2 \tau \,\left[p\, T_1\big(  u(t)\big)(dx)+(1-p) \, T_2\big(  u(t)\big)(dx)\right]  -  2 \tau \, u(t, dx), 
	\end{equation}
	with 
	\begin{equation} \label{5.2}
		u(0,dx)=\phi(dx) \in \mathcal{M}_+ (I).  
	\end{equation}
	
	%\begin{remark}
	%Another way of look the problem would be to consider a two groups model. That is,
	%		\begin{equation} 
		%			\begin{array}{l}
			%		\partial_t u(t, dx) =2 \tau \,\dfrac{B( u, u+v)(dx)}{\int_{I} (u+v)(t,dx)} -  2 \tau \, u(t, dx), \vspace{0.2cm}\\
			%		\partial_t v(t, dx) =2 \tau \,\dfrac{B( v, u+v)(dx)}{\int_{I} (u+v)(t,dx)} -  2 \tau \, v(t, dx), 
			%			\end{array}		
		%	\end{equation}
	%	with 
	%	\begin{equation} \label{5.2}
		%		u(0,dx)=\phi(dx) \in \mathcal{M}_+ (I), 	v(0,dx)=\psi(dx) \in \mathcal{M}_+ (I).  
		%	\end{equation}
	%\end{remark}
	In Figures \ref{Fig3}-\ref{Fig5}, we run an individual based simulation of the model \eqref{5.1}-\eqref{5.2}. Such simulations are stochastic. We first choose a pair randomly following an exponential law with average $1/\tau$. Then we choose the RH model with a probability $p$ and the SN model with a probability $1-p$. Then we apply the transfers rule described in section \ref{Section3}. To connect this problem with our description in the space of measures, we can consider an initial distribution that is a sum of Dirac masses. 
	$$
	\phi(dx) =\sum_{i=1}^{N} \delta_{x_i}(x),
	$$
	in which $x_i$ is the value of the transferable quantity for individual $i$ at $t=0$. 
	When the number of individuals becomes infinite (while keeping a fixed global wealth), the number of meeting events in any time interval becomes infinite and we expect to recover exactly the deterministic model \eqref{1.1} for the distribution of wealth. This is achieved by diminishing the weight of each individual in the population (i.e. replacing $\delta_{x_i}(x)$ by $\frac{1}{N}\delta_{x_i}(x)$). Note that it is not forbidden to have several individuals with the same wealth, and as such, populations consisting of sums of Dirac masses can be achieved in the limit.

	These simulations are, in some sense, a continuous-time version of the computations described in \cite{Chatterjee-Chakrabarti-2007}.
	\begin{figure}
		\begin{center}
			\textbf{(a)} \hspace{5cm} \textbf{(b)}\\
			\includegraphics[scale=0.12]{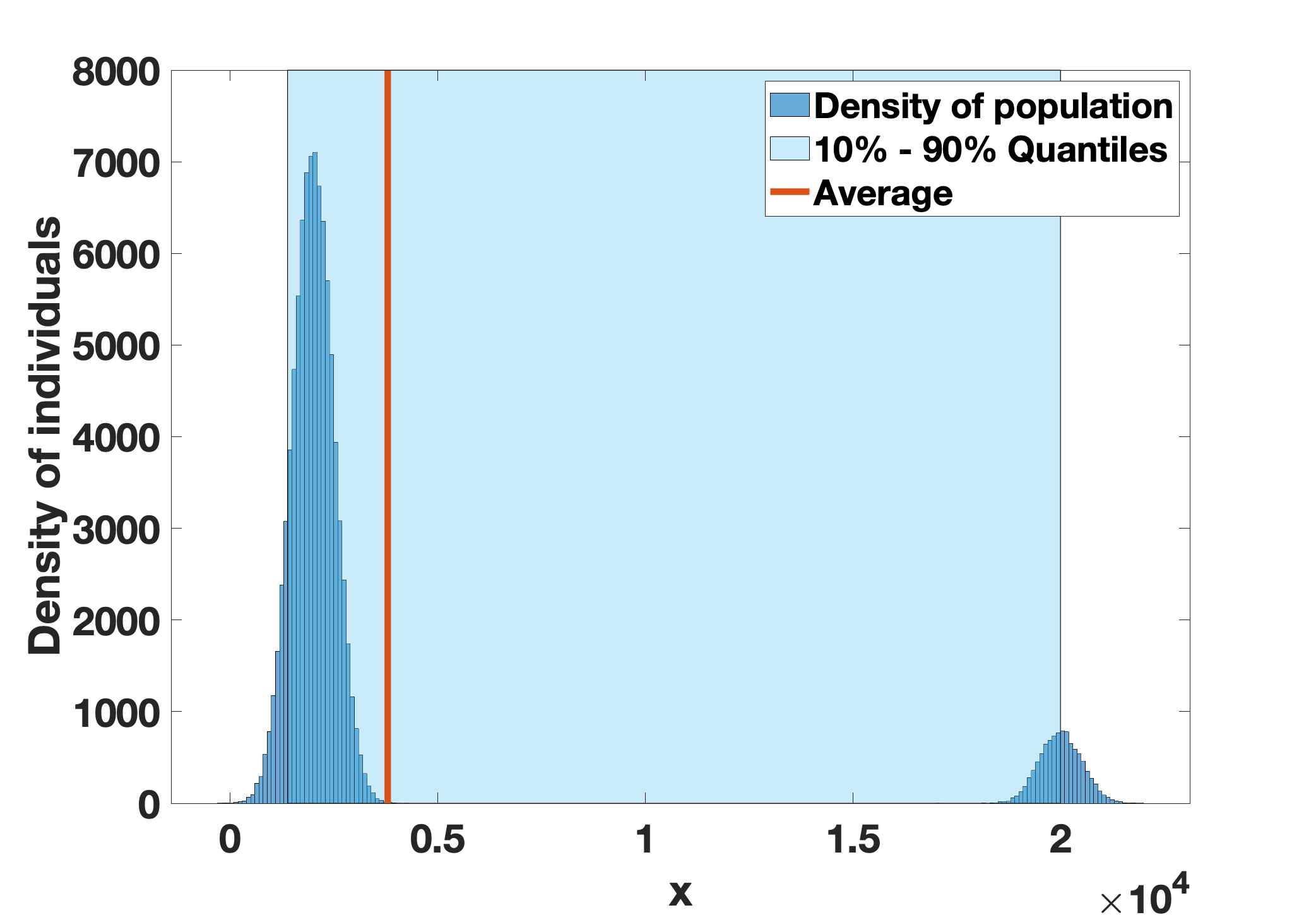} \hspace{1cm}
			\includegraphics[scale=0.12]{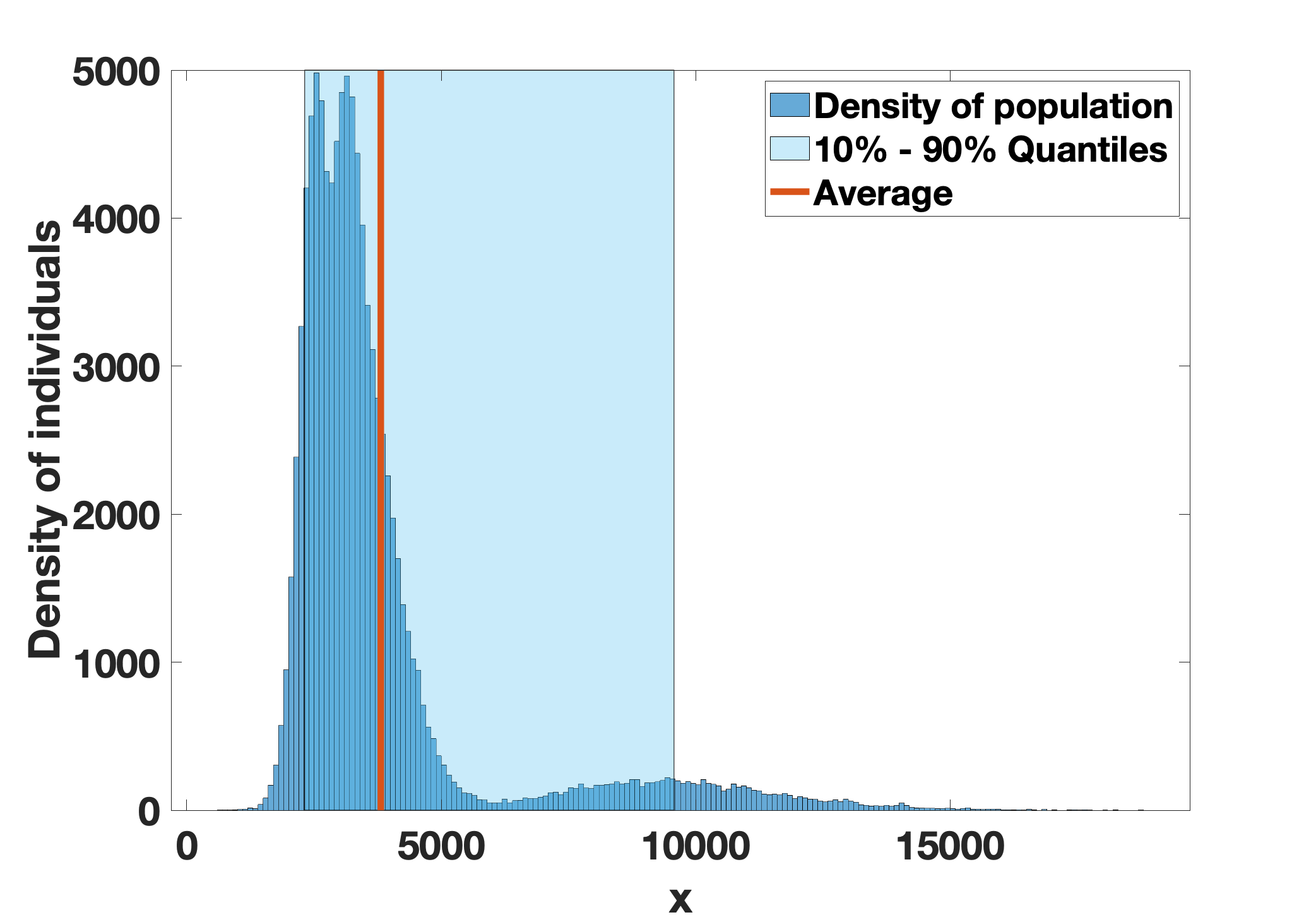}\\
			\textbf{(c)} \hspace{5cm} \textbf{(d)}\\
			\includegraphics[scale=0.12]{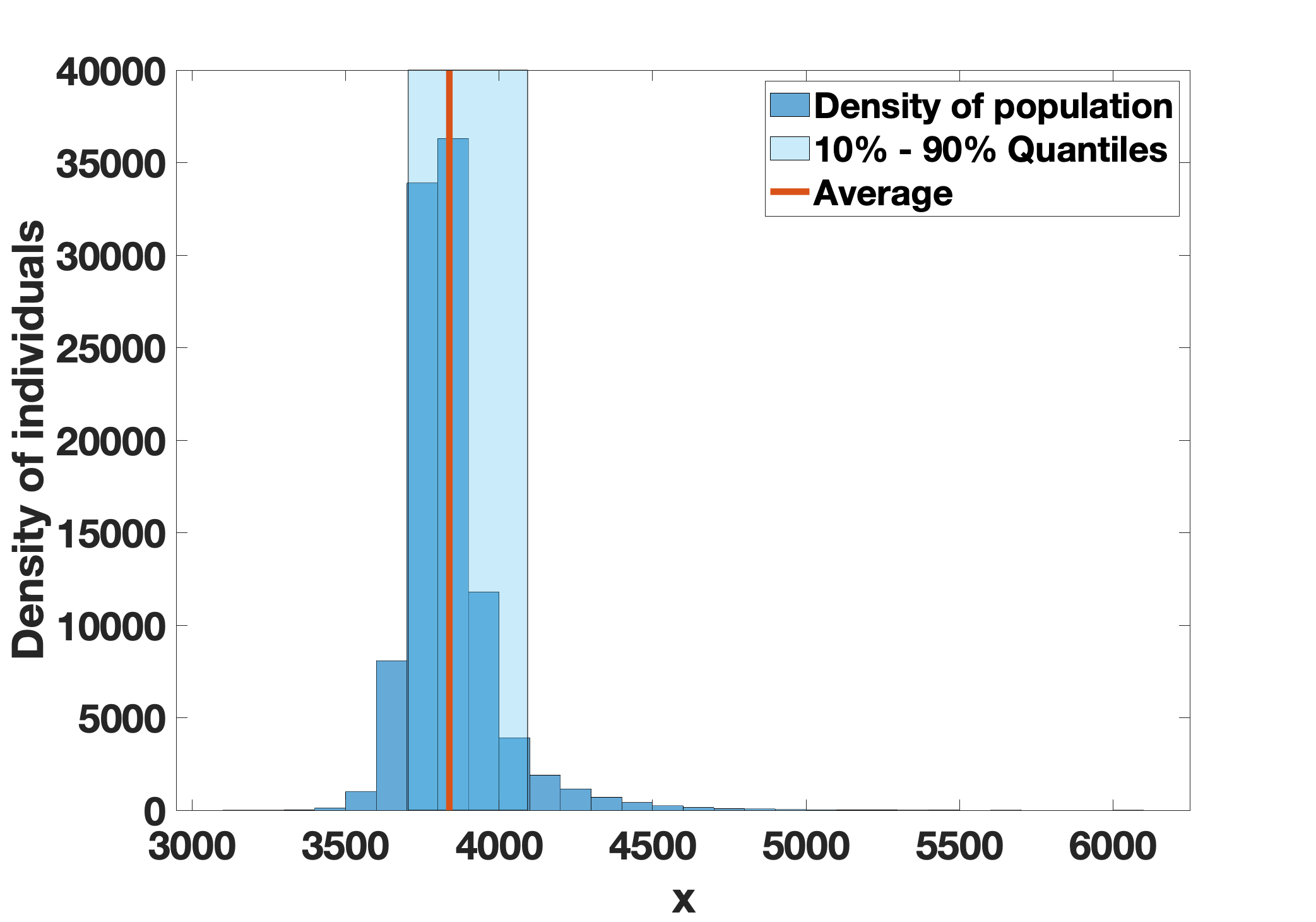} \hspace{1cm}
			\includegraphics[scale=0.12]{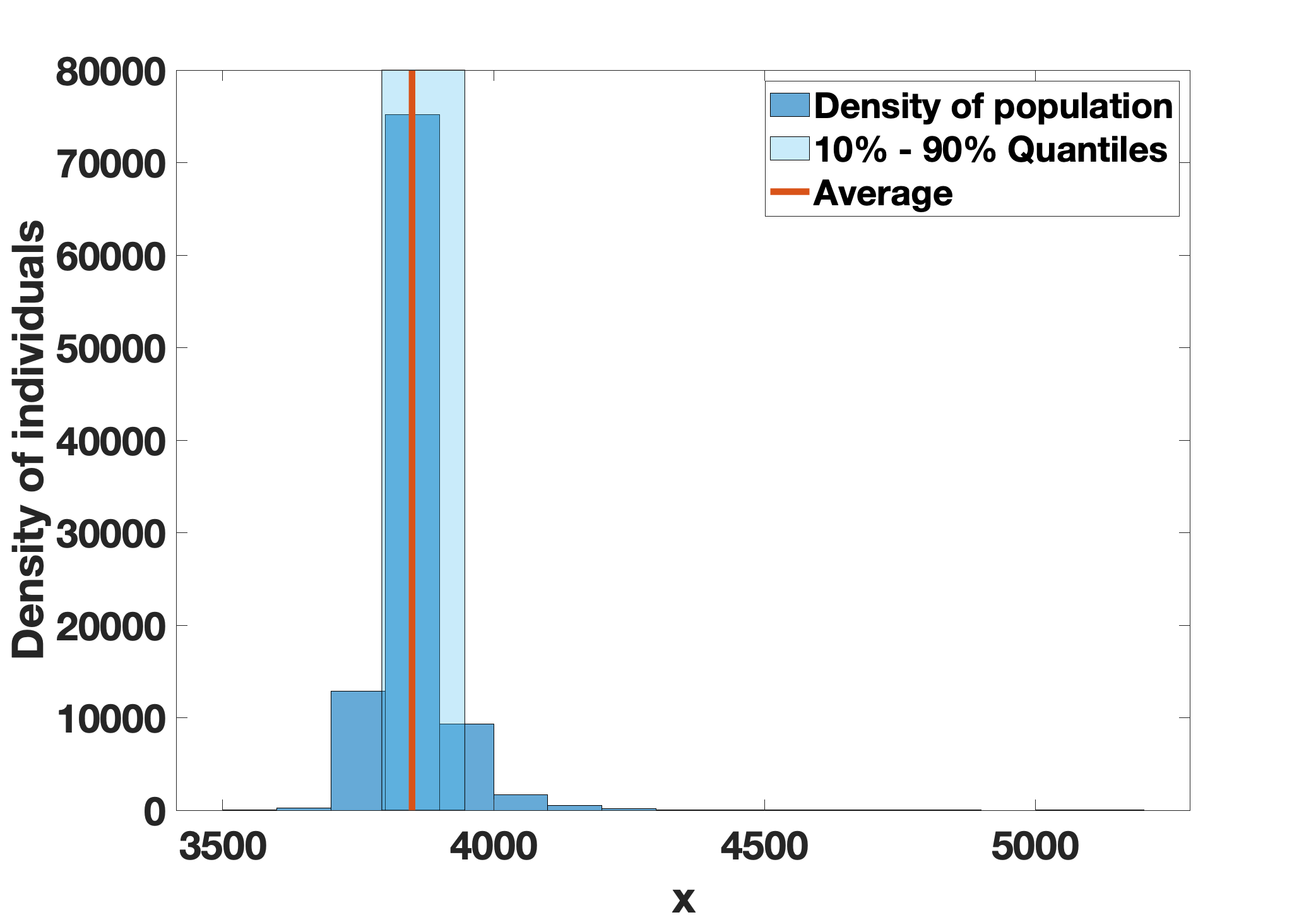}
		\end{center}
		\caption{\textit{In this figure, we use $p=1$ (i.e.   $100 \%$ RH model),  $f_1=f_2=0.1$, $1/\tau=1$ years. We start the simulations with $100 \, 000$ individuals. The figures (a) (b) (c) (d) are respectively the initial distribution at time $t=0$, and the distribution  $10$ years, $50$ years and $100$ years.    }}\label{Fig3}
	\end{figure}
	
	\begin{figure}
		\begin{center}
			\includegraphics[scale=0.12]{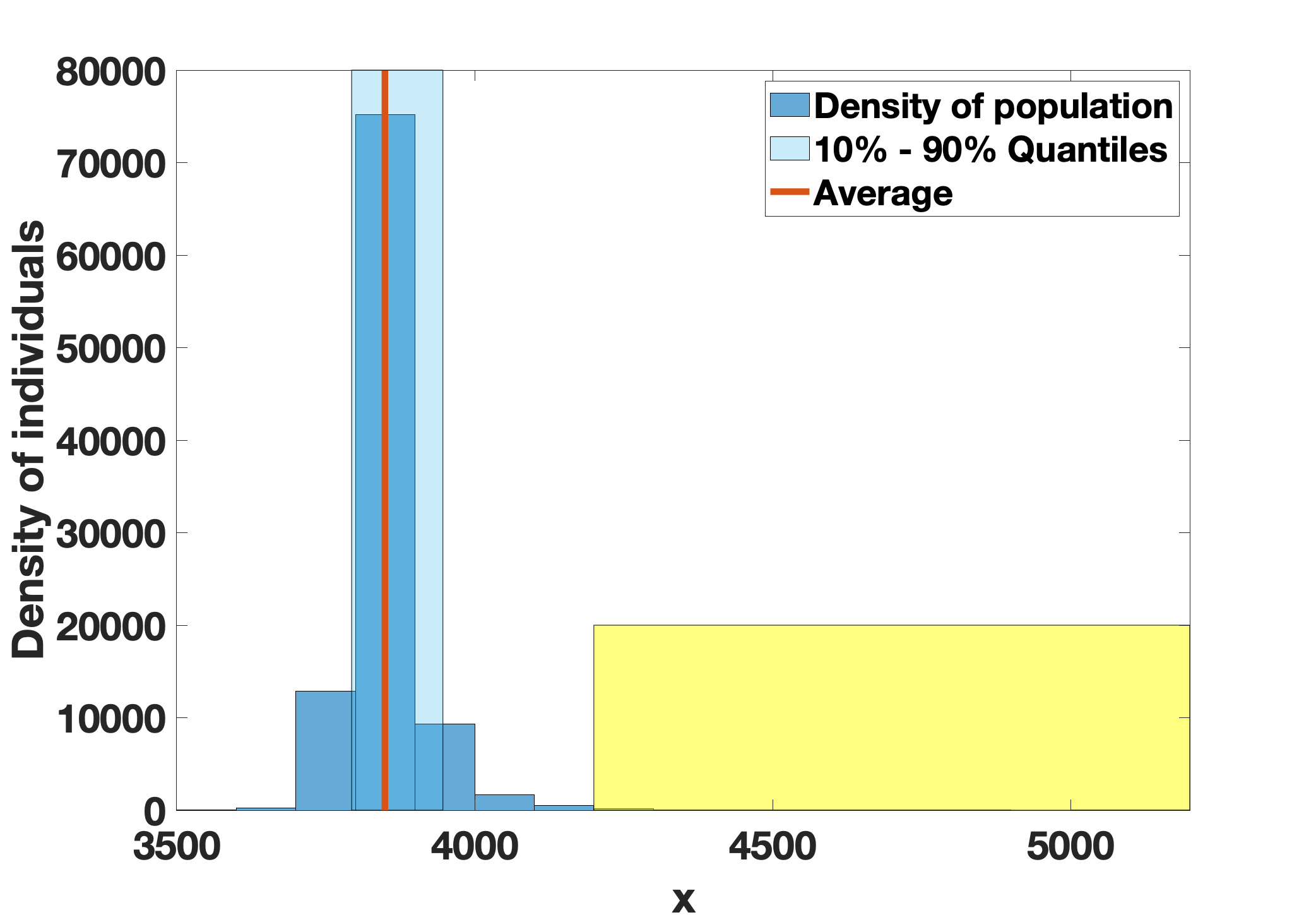} \hspace{1cm}
			\includegraphics[scale=0.12]{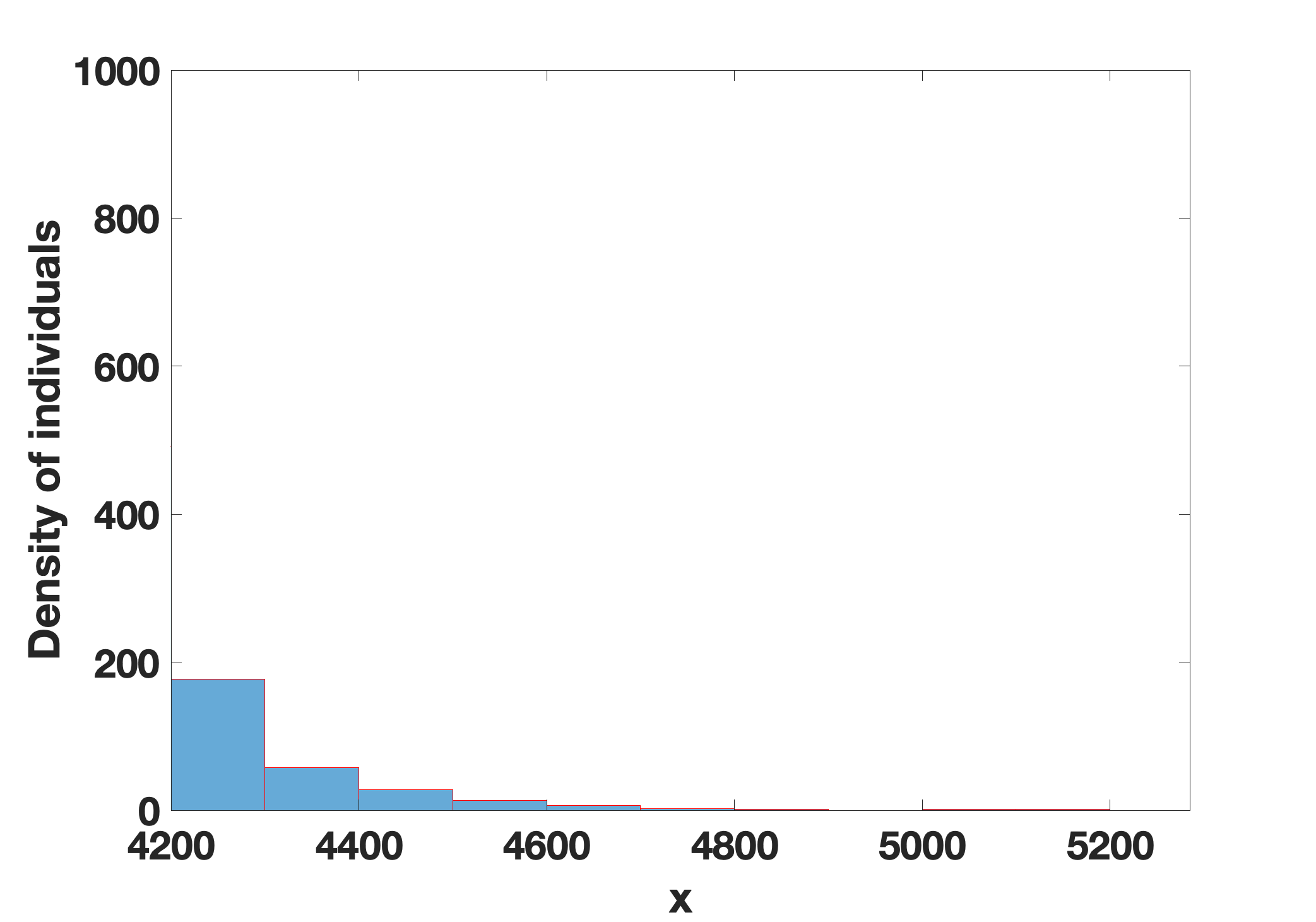}\\
		\end{center}
		\caption{\textit{In this figure, we zoom on the distribution for $t=100$ in Figure \ref{Fig3} (d).  The figure on the right-hand side corresponds to the yellow region in the left figure.  }}\label{Fig3A}
	\end{figure}

	\begin{figure}
		\begin{center}
			\textbf{(a)} \hspace{5cm} \textbf{(b)}\\
			\includegraphics[scale=0.12]{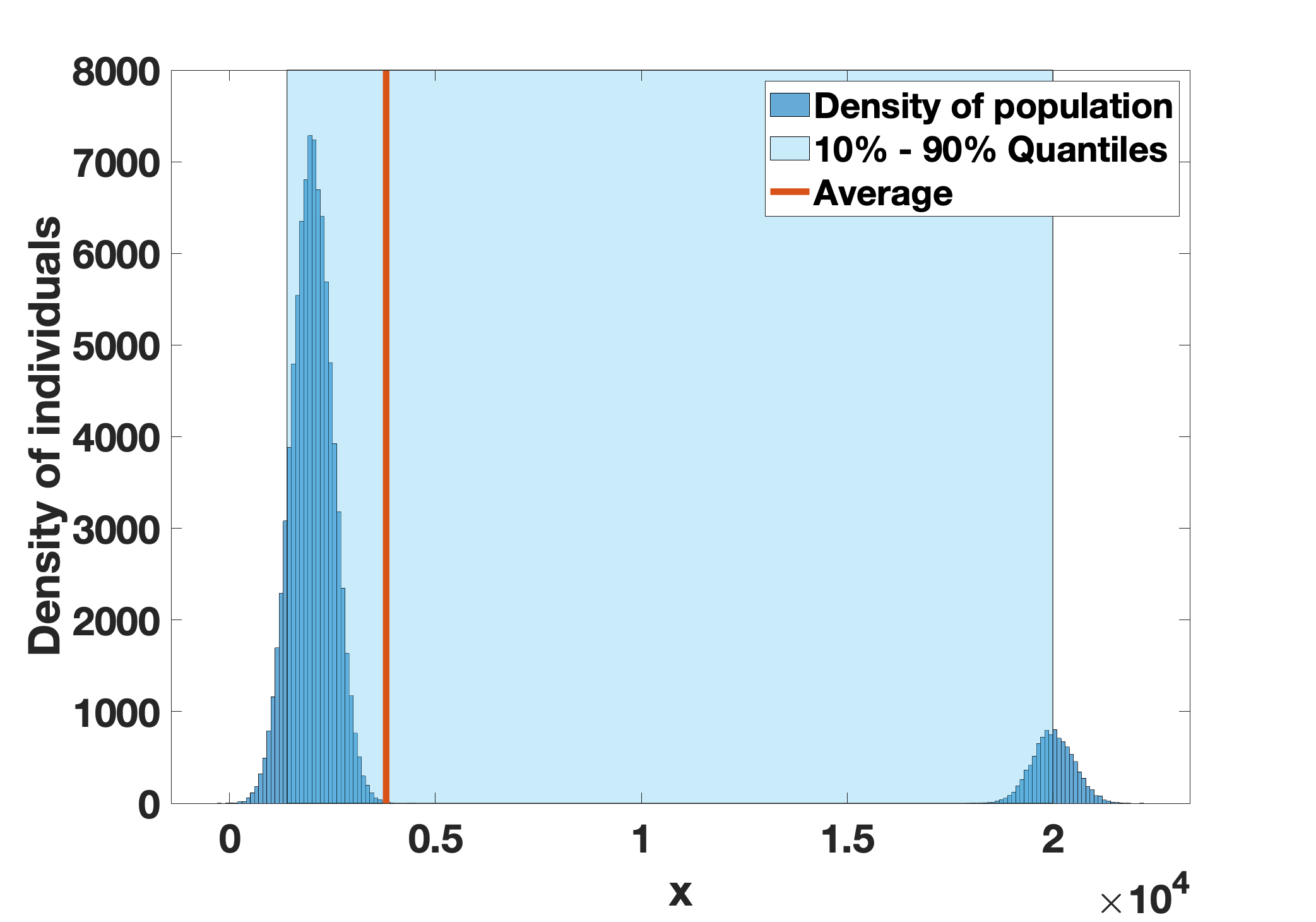}  \hspace{1cm}
			\includegraphics[scale=0.12]{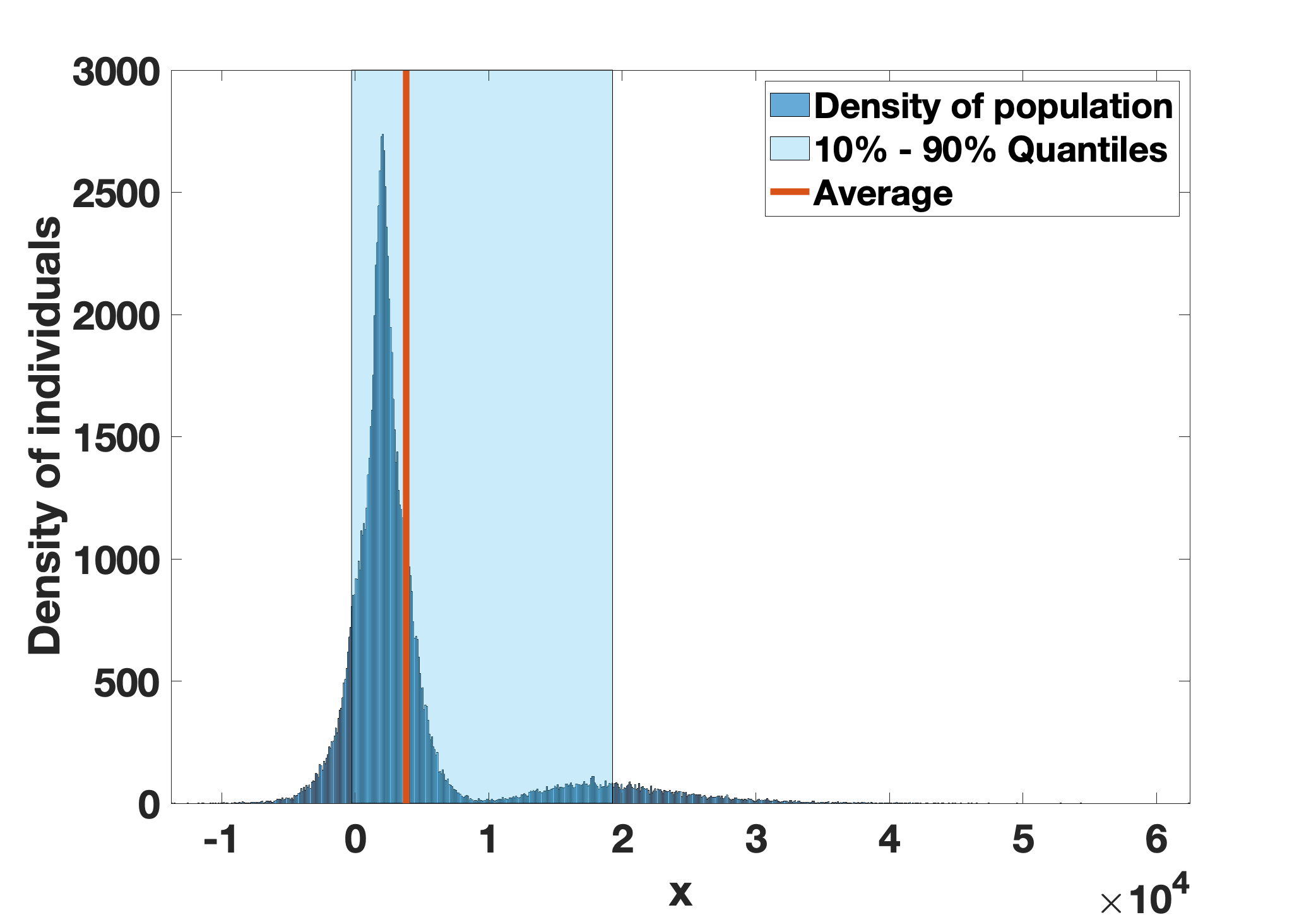}\\
			\textbf{(c)} \hspace{5cm} \textbf{(d)}\\
			\includegraphics[scale=0.12]{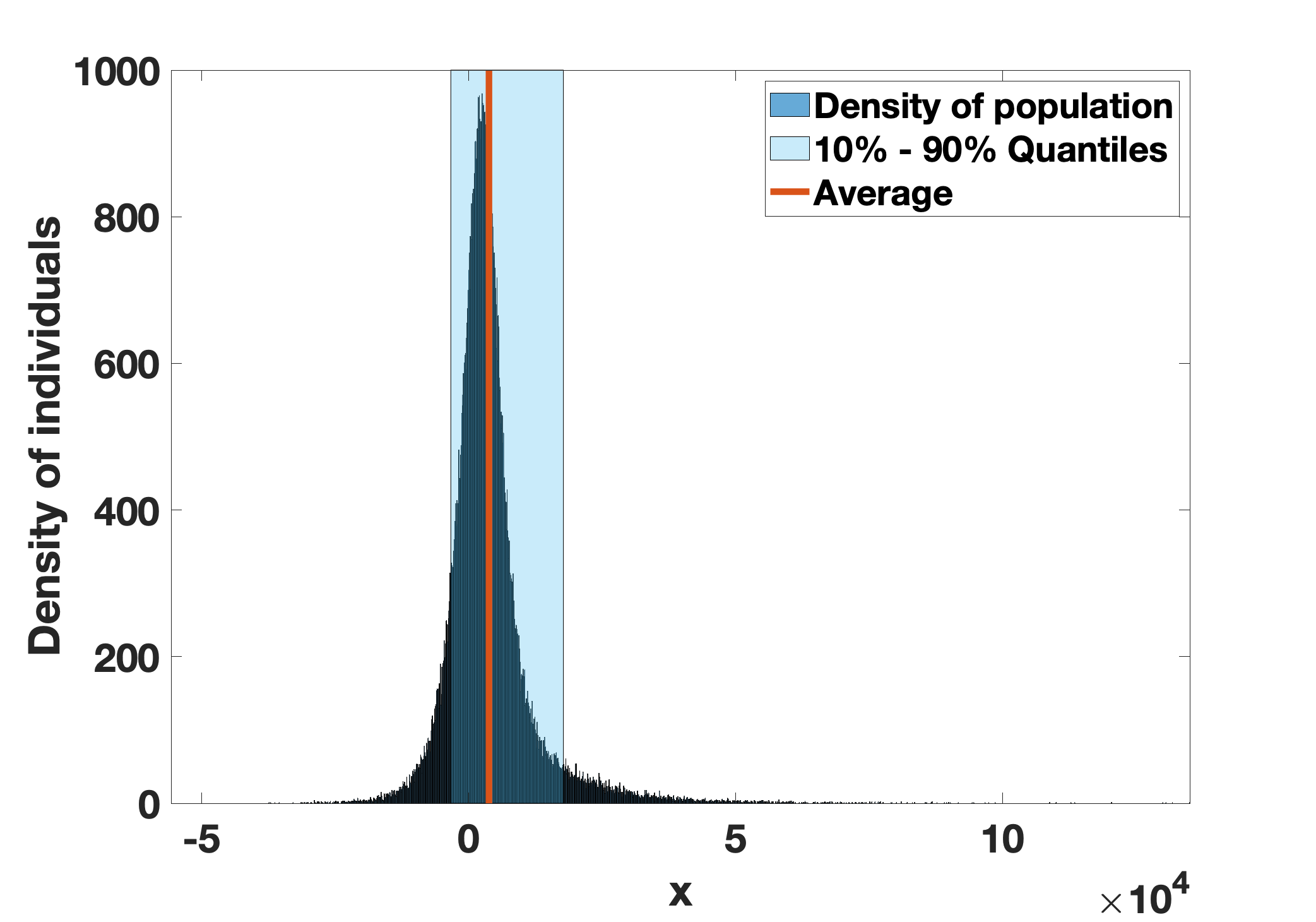}  \hspace{1cm}
			\includegraphics[scale=0.12]{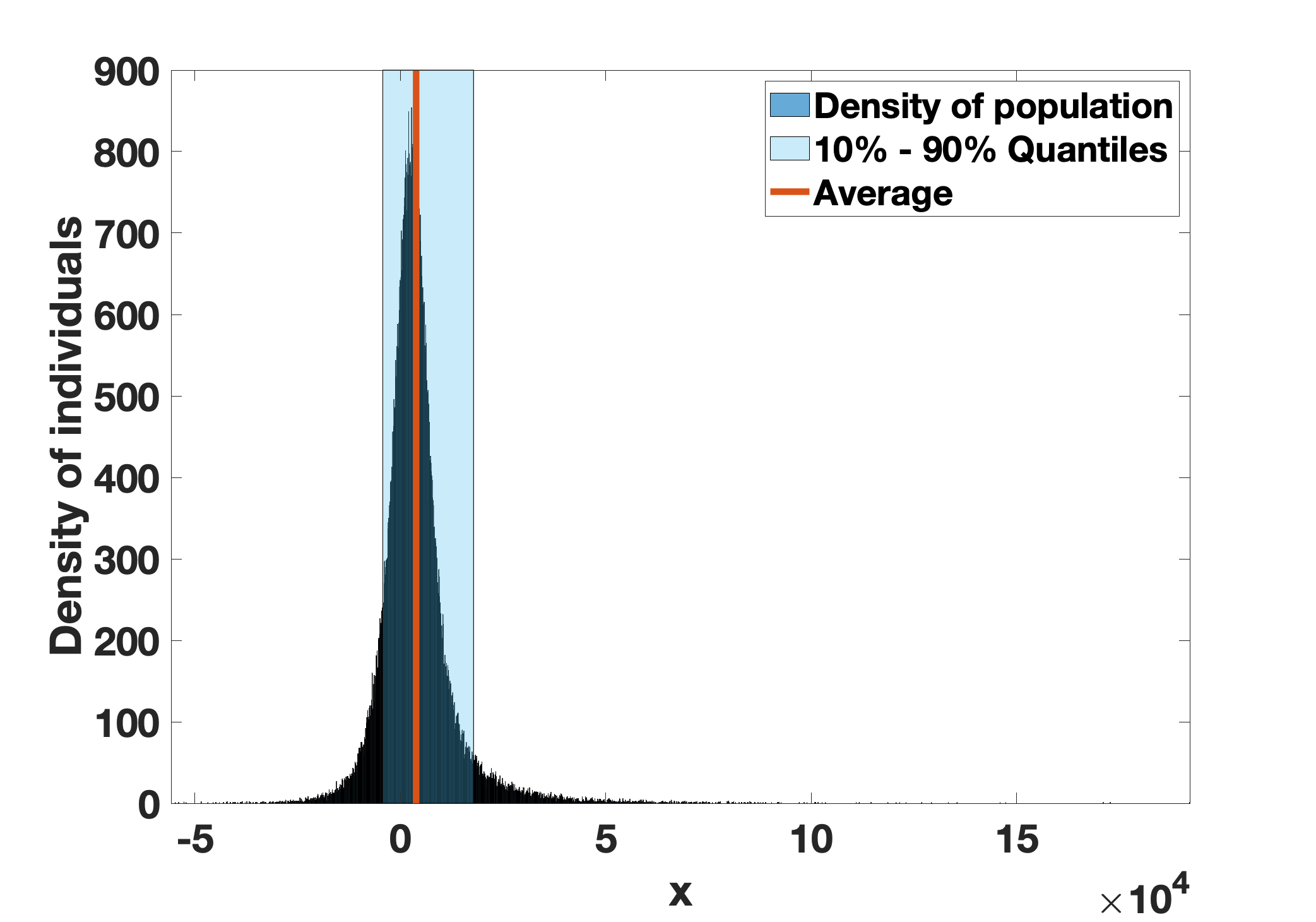}
		\end{center}
		\caption{\textit{In this figure, we use $p=0.5$ (i.e.   $50 \%$ RH model and $50\%$ SN model),  $f_1=f_2=0.1$, $1/\tau=1$ years. We start the simulations with $100 \, 000$ individuals. The figures (a) (b) (c) (d) are respectively the initial distribution at time $t=0$, and the distribution  $10$ years, $50$ years and $100$ years.     }}\label{Fig4}
	\end{figure}
	
	\begin{figure}
		\begin{center}
			\includegraphics[scale=0.12]{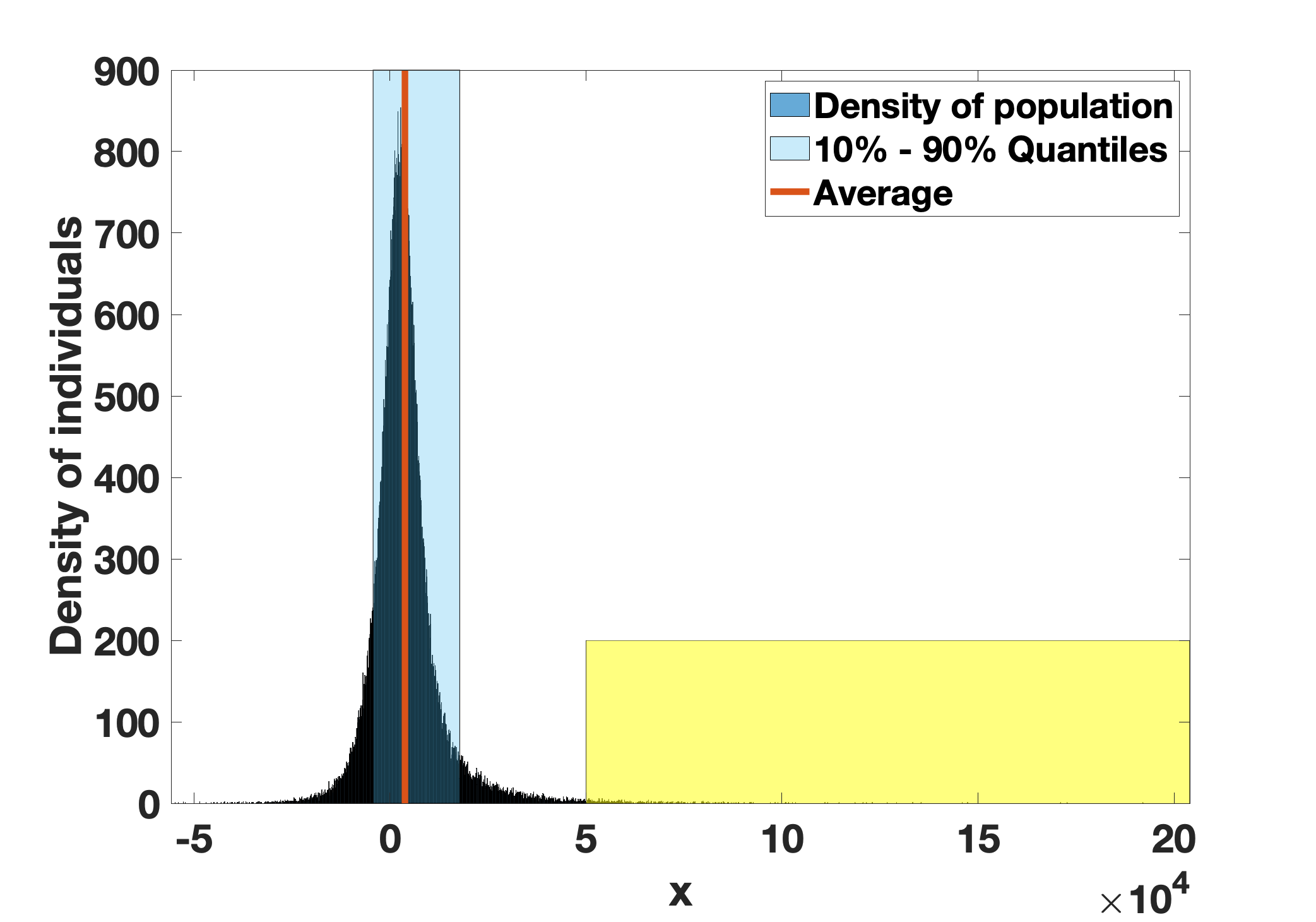} \hspace{1cm}
			\includegraphics[scale=0.12]{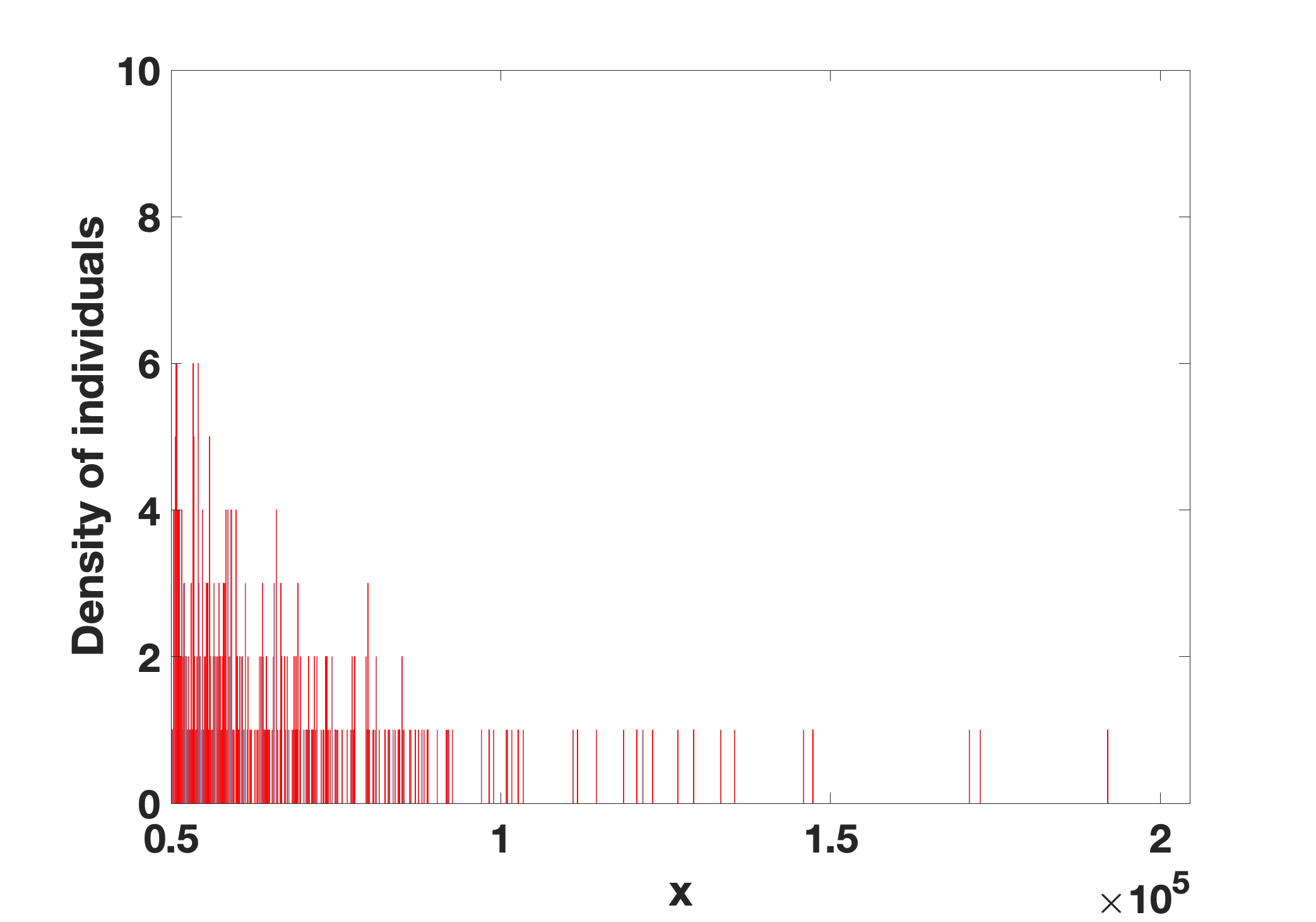}\\
		\end{center}
		\caption{\textit{In this figure, we zoom on the distribution for $t=100$ in Figure \ref{Fig4} (d).  The figure on the right-hand side corresponds to the yellow region in the left figure.  }}\label{Fig4A}
	\end{figure}
	
	\begin{figure}
		\begin{center}
			\textbf{(a)} \hspace{5cm} \textbf{(b)}\\
			\includegraphics[scale=0.12]{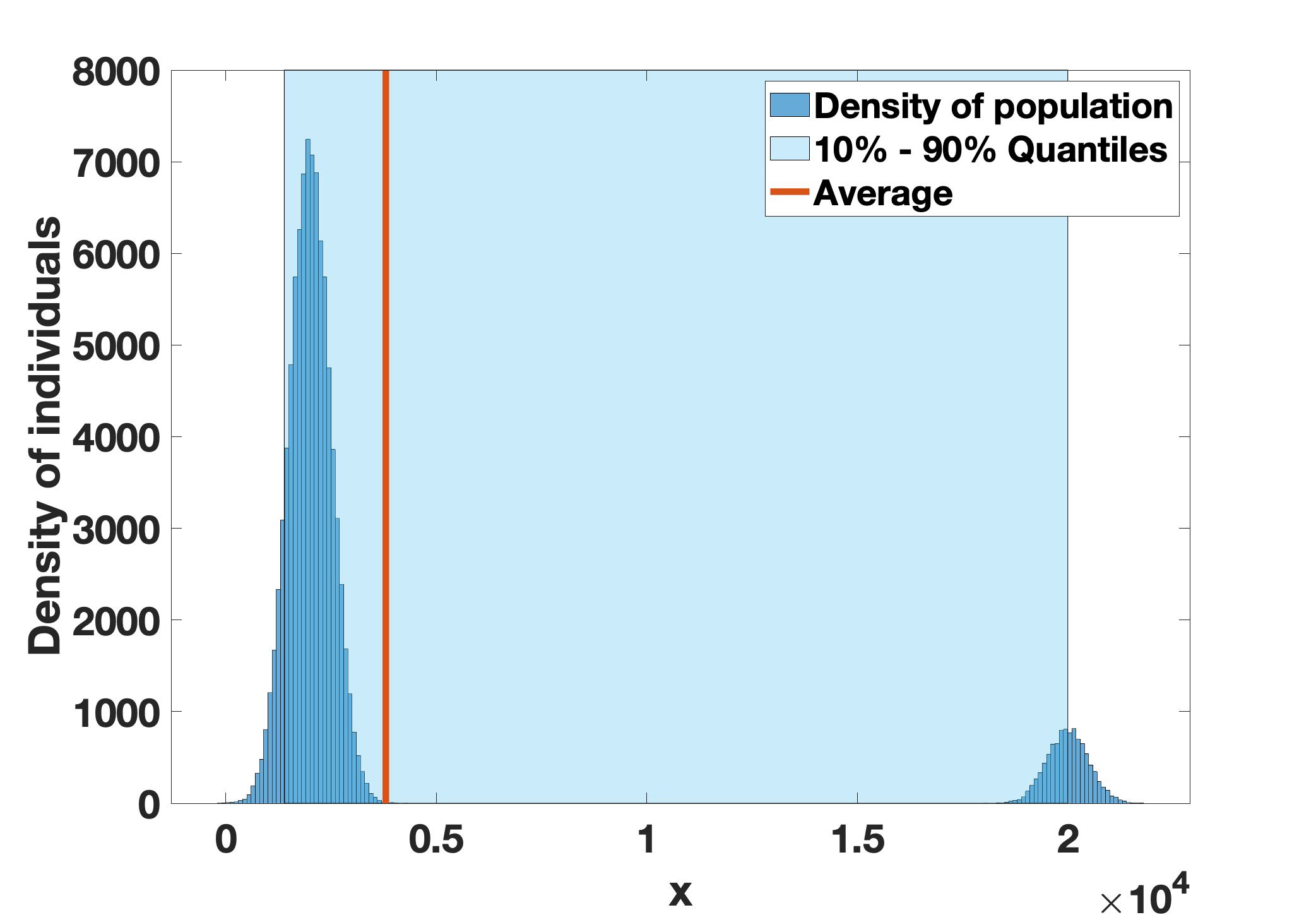}  \hspace{1cm}
			\includegraphics[scale=0.12]{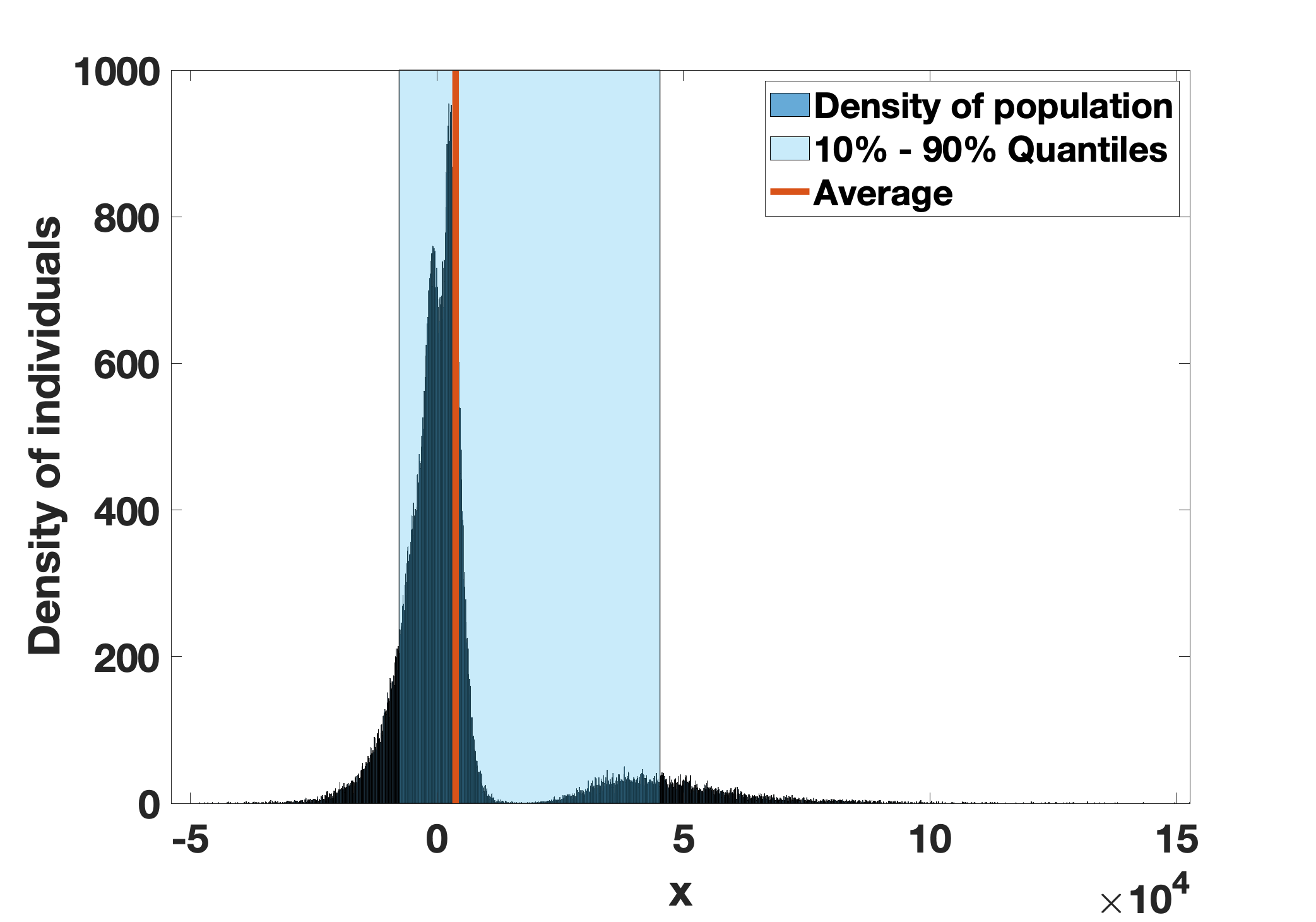}\\
			\textbf{(c)} \hspace{5cm} \textbf{(d)}\\
			\includegraphics[scale=0.12]{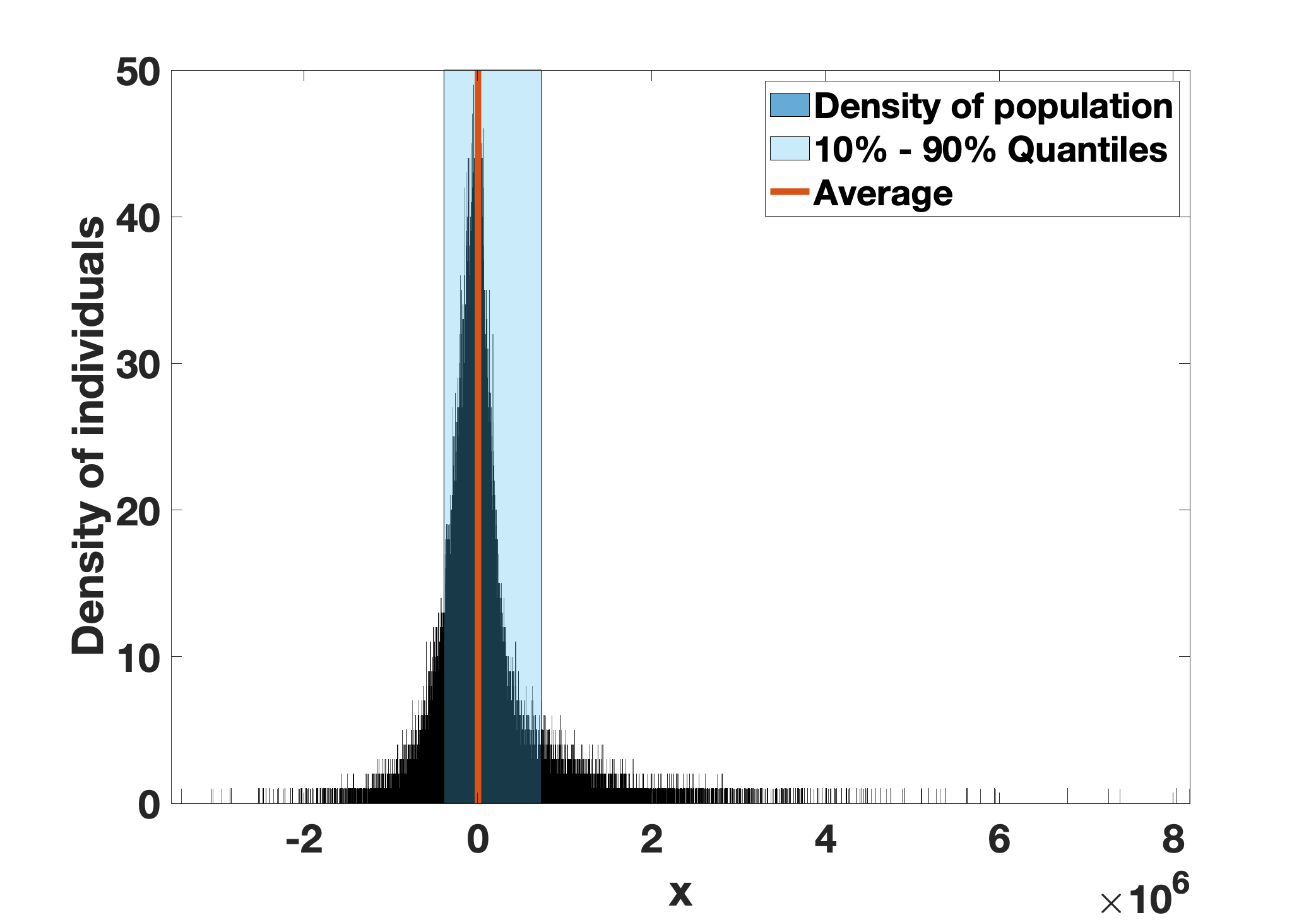}  \hspace{1cm}
			\includegraphics[scale=0.12]{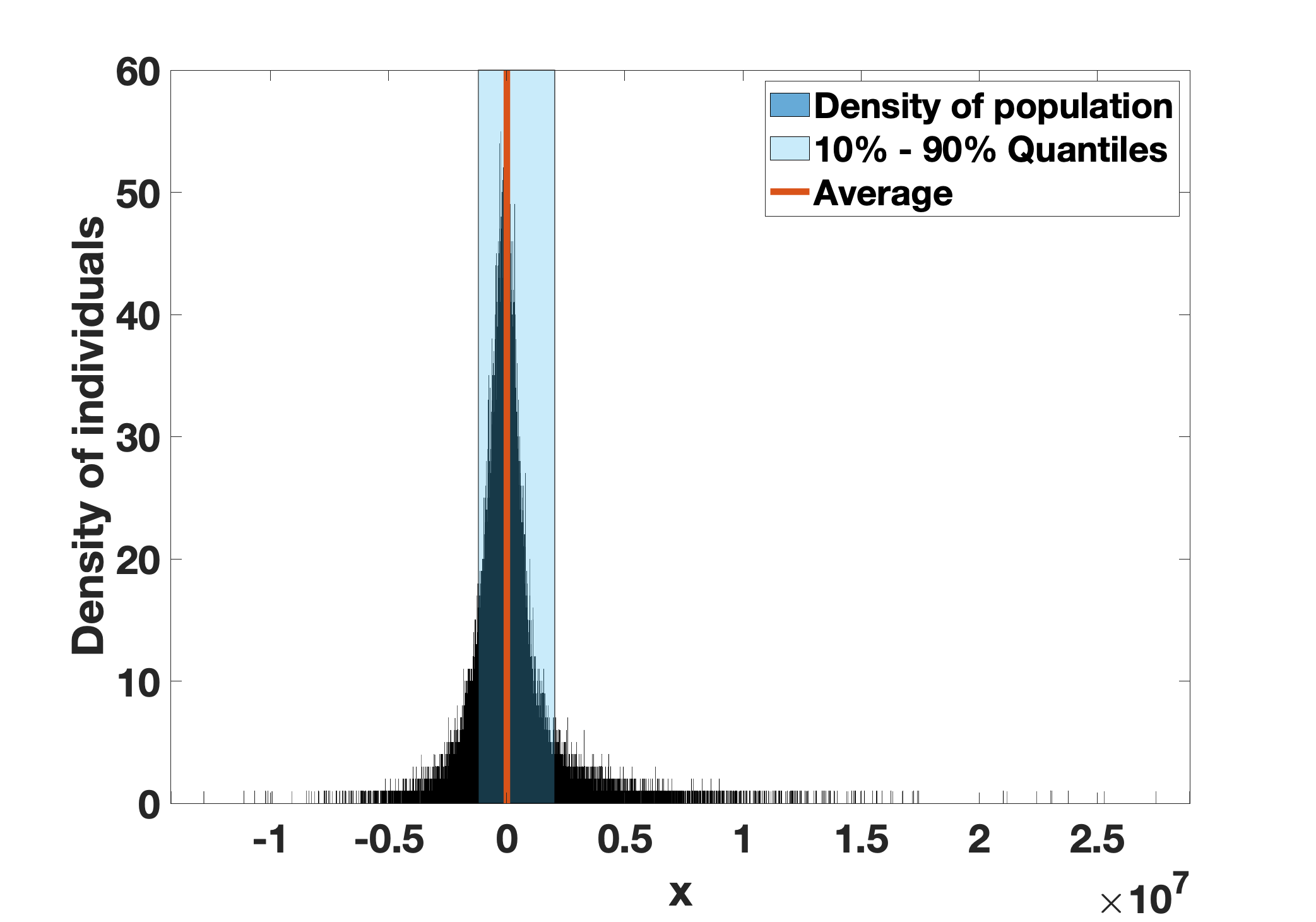}
		\end{center}
		\caption{\textit{In this figure, we use $p=0$ (i.e.   $0 \%$ RH model and $100\%$ SN model),  $f_1=f_2=0.1$, $1/\tau=1$ years. We start the simulations with $100 \, 000$ individuals. The figures (a) (b) (c) (d) are respectively the initial distribution at time $t=0$, and the distribution  $10$ years, $50$ years and $100$ years.     }}\label{Fig5}
	\end{figure}

	\begin{figure}
		\begin{center}
			\includegraphics[scale=0.12]{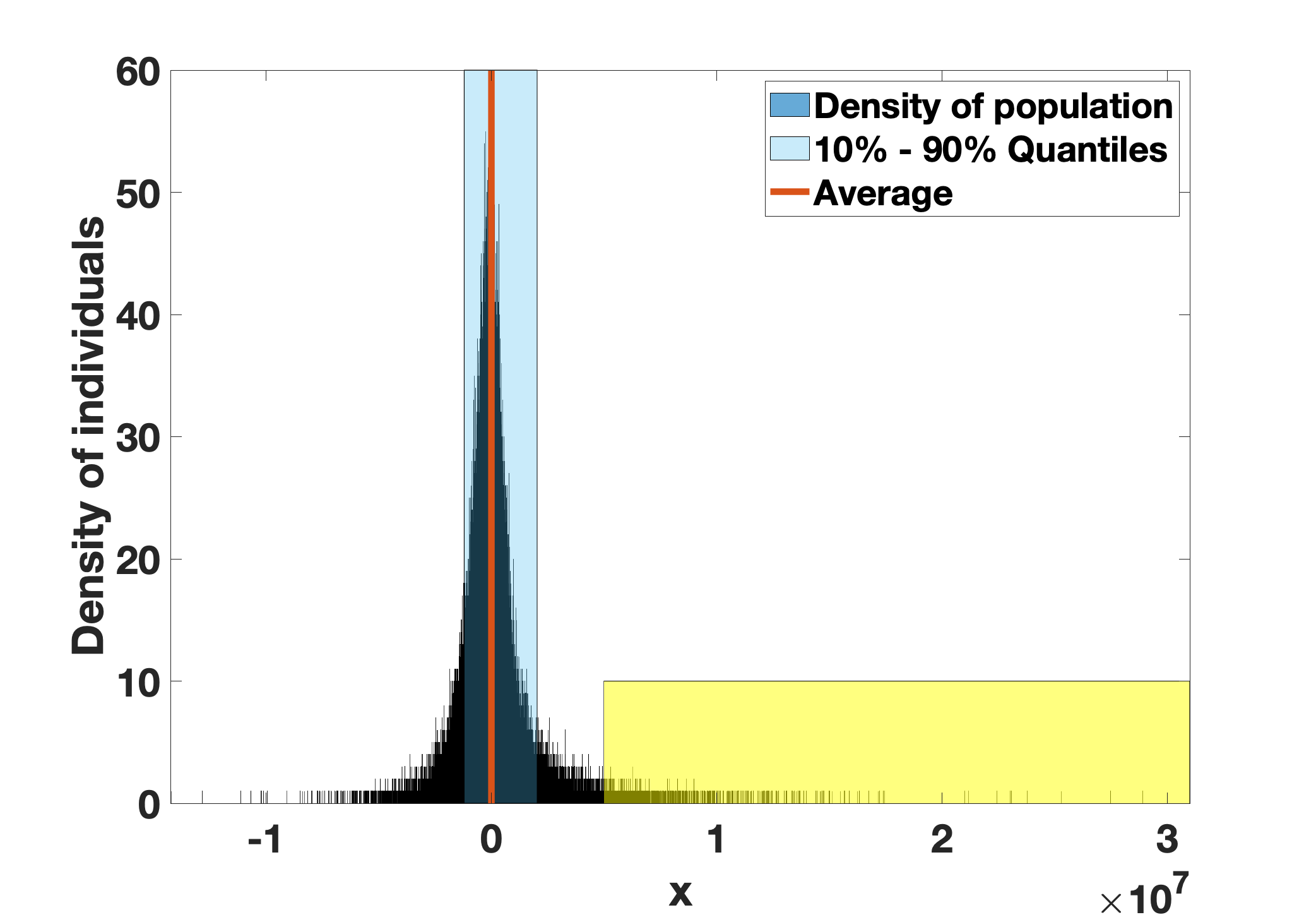} \hspace{1cm}
			\includegraphics[scale=0.12]{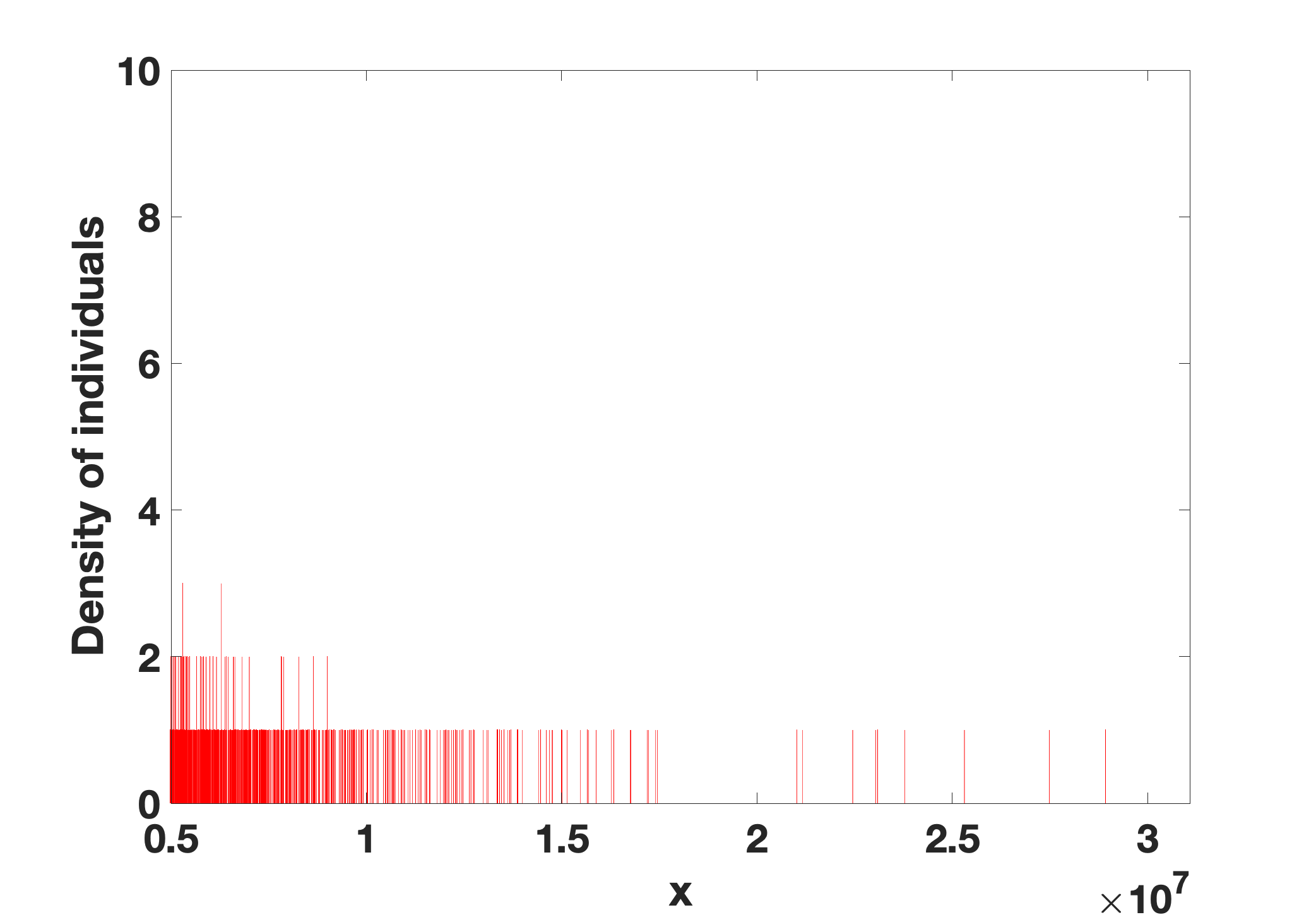}\\
		\end{center}
		\caption{\textit{In this figure, we zoom on the distribution for $t=100$ in Figure \ref{Fig5} (d).  The figure on the right-hand side corresponds to the yellow region in the left figure.  }}\label{Fig5A}
	\end{figure}
	
	Figure  \ref{Fig3} corresponds to the full RH model which corresponds to $p=1$. In that case, the population density converges to a Dirac mass centered at the mean value. That is, everyone will ultimately have the same amount of transferable quantity.
	
	\medskip 
	Whatever the value of $p$ strictly less than $1$, the simulations can be summed up by saying that ``there is always a sheriff in town''. In  Figures \ref{Fig4}-\ref{Fig5}, the unit for $x$-axis changes from (a) to (d).  We can see from that some rich guys will always become richer and richer. The SN model induces competition between the poorest individuals also, and the population ends up after 100 years with a lot of debts. In other words, the richest individuals are becoming richer, while the poorest are becoming poorer. The effect in changing the value of the parameter  $p \in [0,1)$ is strictly positive, it seems that it is only a matter of time before we end up with a very segregated population.  We observe a difference for the richest of two orders of magnitude between the case $p=0.5$ and $p=0$. We conclude by observing that the smaller $p$ is, the more the wealthiest individuals are rich.
	This observation has another, slightly philosophical interpretation, that enormous amounts of wealth can be concentrated within a very limited number of individuals in a very simple random exchanges model, as soon as the exchanges have the slightiest bias towards the rich. We end up with a situation in which most of the population is very poor while a few individuals are very rich. This catastrophic state is due to an advantage in trading based only on the initial wealth of each individual.
	
	\vspace{1cm}
	
	\medskip 	
	\noindent \textbf{Data availability:} No data were produced for this study. 
	
	\medskip 
	\noindent \textbf{Conflict of interest:} The authors declare no conflict of interest. 

	\bibliographystyle{plain}
\bibliography{sn-bibliography}

\includepdf[pages=-]{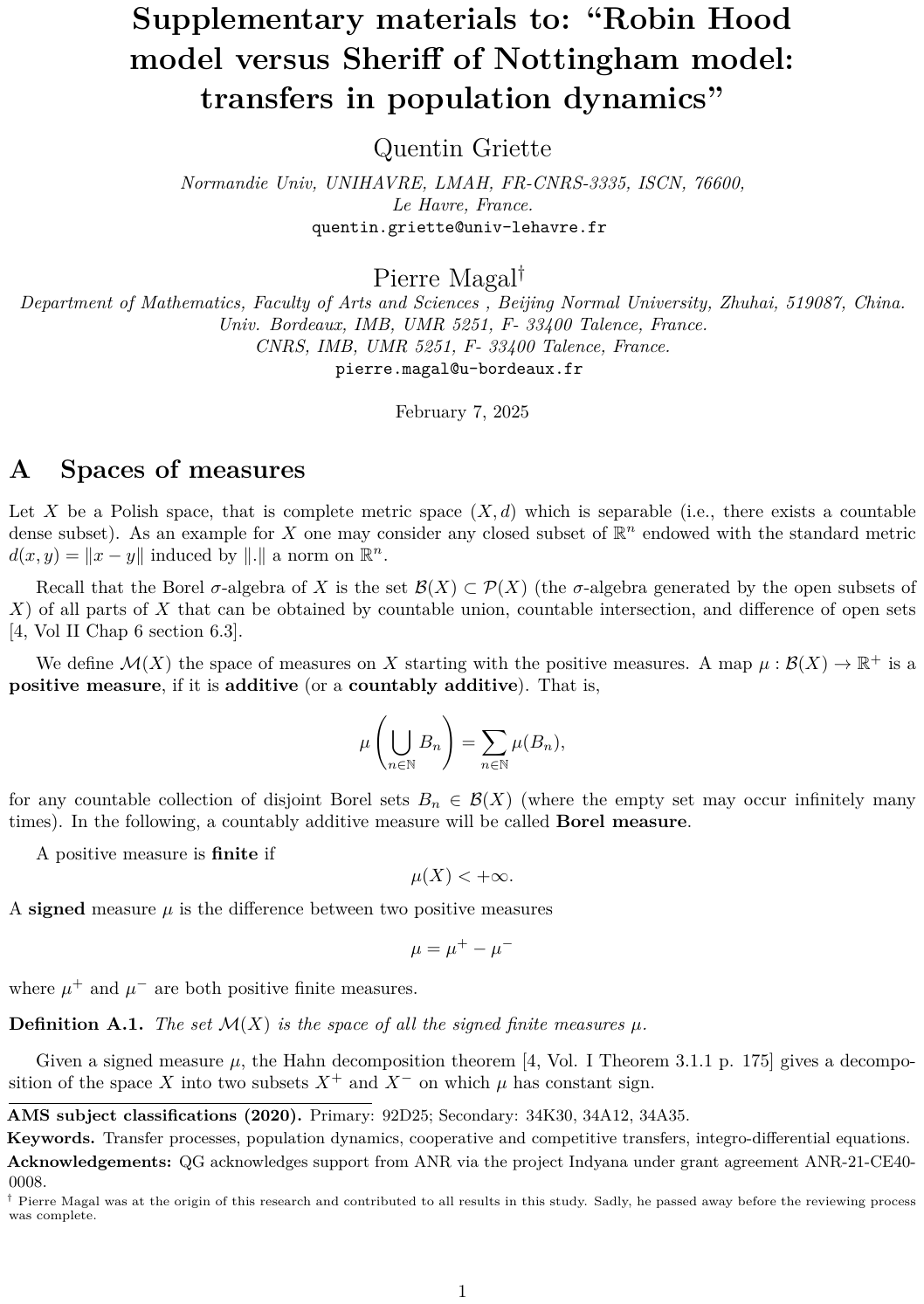}
\end{document}